\documentclass[a4paper,12pt]{article}
\usepackage{amsmath}
\usepackage{amsthm}
\usepackage[dvips]{graphicx}
\usepackage{epsfig}

\usepackage{amssymb}
\usepackage{latexsym}

\usepackage{amssymb,amsmath,amsthm,amscd,latexsym}

 \theoremstyle{plain}
 \newtheorem{theo}{Theorem}[section]
 \newtheorem{lem}[theo]{Lemma}
 \newtheorem{corol}[theo]{Corollary}
 \newtheorem{prop}[theo]{Proposition}
 
 \newtheorem{aff}[theo]{Assertion}

\theoremstyle{definition}
 \newtheorem*{defi}{Definition}

 \newtheorem*{rem}{Remark}
\def\er{\operatorname{rad}}
\def\ab{\operatorname{ab}}
\def\tr{\operatorname{triv}}

\let\bord\partial
\let\bydef\emph
\let\epsi\varepsilon

\def\Rr{\mathbf{R}}

\newcommand{\ZZ} {{\mathbb{Z}}}

\newcommand{{\II}}{{\mathbb{II}}}

\def\inj{\mathop{\rm inj}\nolimits}
\def\diam{\mathop{\rm diam}\nolimits}

\def\Ric{\mathop{\rm Ric}\nolimits}
\def\Rm{\mathop{\rm Rm}\nolimits}

\def\vol{\mathop{\rm vol}\nolimits}

\def\var{manifold}
\def\irr{ir\-re\-du\-ci\-ble}

\def\er{\operatorname{rad}}
\def\ab{\operatorname{ab}}

\begin{document}

\title{Weak collapsing and geometrisation of aspherical 3-manifolds}
\author{L.~Bessi\`eres, G.~Besson, M.~Boileau, S.~Maillot, J.~Porti}
\date{\today}

\maketitle

\abstract{Let $M$ be a closed, orientable, irreducible, non-simply connected 3-manifold.
We prove that if $M$ admits a sequence of Riemannian metrics whose sectional curvature is locally controlled and whose thick part becomes asymptotically hyperbolic and has a sufficiently small volume, then $M$ is Seifert fibred or contains an incompressible torus. This result gives an alternative approach for the last step in Perelman's proof of the Geometrisation Conjecture for aspherical 3-manifolds.

\section*{Introduction}

Thurston's Geometrisation Conjecture states that  any closed, o\-rien\-ta\-ble, irreducible 3-dimensional manifold $M$  is hyperbolic, Seifert fibred, or contains an incompressible torus.
This conjecture has been proved recently by G. Perelman~\cite{PerelmanI,PerelmanII,PerelmanIII}
(see 
also~\cite{KleinerLott,MorganTian,CaoZhu}) using R.~Hamilton's Ricci flow. In this paper, we shall be concerned with the case
where $\pi_1 M$ is nontrivial. Our results apply in particular if $\pi_1M$ is infinite, which under
the above hypotheses is equivalent to $M$ being aspherical.

The last step of Perelman's proof in this case relies on a `collapsing theorem' which is independent of the Ricci
flow part. This result is stated without proof as Theorem~7.4 in~\cite{PerelmanII}. A version of this theorem for closed $3$-manifolds
is given in the appendix of \cite{ShioyaYamaguchi} using deep
results  of Alexandrov space theory, including Perelman's stability theorem (see \cite{Kapovitch}) and a fibration theorem for Alexandrov spaces~\cite{Yamaguchi}.

Our main result, Theorem~\ref{TH:GRAPH} below, implies Theorem\ 7.4 of~\cite{PerelmanII} for closed,
irreducible $3$-manifolds which are not simply connected, and is sufficient to apply Perelman's construction of Ricci flow with surgery to geometrise these manifolds. 
The proof of Theorem~\ref{TH:GRAPH} combines arguments from Riemannian geometry, algebraic topology, and $3$-manifold
theory. It uses Thurston's hyperbolisation theorem for Haken manifolds, but avoids the stability and fibration theorems for Alexandrov spaces.

In this text we call  \bydef{hyperbolic manifold} a complete $3$-manifold with constant sectional 
curvature equal to $-1$ and \emph{finite volume.} The hyperbolic metric, which is unique up to isometry by Mostow rigidity, will be denoted by $g_{hyp}$.

In the next two definitions, $M$ is a $3$-manifold.
\begin{defi}
Let $g$ be a Riemannian metric on $M$ and $\varepsilon>0$ be a real number. A point $x\in M$
is \bydef{$\varepsilon$-thin with respect to $g$} if  there exists $0<\rho\leq 1$ such that on the ball  $B(x,\rho)$, the sectional curvature is greater than or equal to $-\rho^{-2}$ and the volume of this ball is less than $\varepsilon\, \rho^3$.
Otherwise we say that $x$ is \bydef{$\varepsilon$-thick with respect to $g$}. The set of $\varepsilon$-thin points (resp.~$\varepsilon$-thick points) is called the \bydef{$\varepsilon$-thin part} (resp.~\bydef{$\varepsilon$-thick part}) of $M$.
\end{defi}

The following is a technical condition which guarantees the regularity of certain
limits of riemannian manifolds.
\begin{defi}
Let $g_n$ be a sequence of Riemannian metrics on $M$. We say that
$g_n$ has  \bydef{locally controlled curvature in the sense of
Perelman} if it has the following property: for all
$\varepsilon>0$ there exist $\overline
r(\varepsilon)>0,K_0(\varepsilon),K_1(\varepsilon)>0$, such that for
$n$ large enough, if $0<r\leq \bar r(\epsilon)$, $x\in (M,g_n)$ satisfies
 $\vol (B(x,r))/r^3\ge \varepsilon $, and the sectional curvature on 
 $B(x,r)$ is greater than or equal to $-r^{-2}$, then $\vert \Rm (x) \vert < K_0
r^{-2}$ and $\vert \nabla \Rm (x) \vert < K_1 r^{-3}$.
\end{defi}

Next we define a topological invariant $V_0(M)$  which is essential to this paper.
Let $M$ be a closed $3$-manifold.
For us, a \bydef{link} in $M$ is a (possibly empty, possibly disconnected) closed $1$-submanifold
of $M$. A link is \bydef{hyperbolic} if its complement is a hyperbolic $3$-manifold.
The invariant $V_0(M)$ is defined as the infimum of the volumes of all hyperbolic links in $M$. This
quantity is finite because any closed $3$-manifold contains a hyperbolic link~\cite{Myers}. Since the set of volumes of hyperbolic $3$-manifolds is well-ordered, this
infimum is always realised by some hyperbolic $3$-manifold $H_0$; in particular, it is positive. We note
that $M$ is hyperbolic if and only if  $H_0=M$;  indeed, every hyperbolic Dehn filling on a hyperbolic
manifold strictly decreases the volume \cite{Bessieres}. 

We now state the main result of this article:

\pagebreak
\begin{theo}\label{TH:GRAPH}
Let $M$ be a non-simply connected, closed, orientable, irreducible $3$-manifold.
Suppose that there exists a sequence $g_n$ of Riemannian metrics on $M$ satisfying:
\begin{itemize}
\item[(1)]  The sequence  $\vol(g_n)$ is bounded.
\item[(2)]  Let $\varepsilon>0$ be a real number and $x_n\in M$ be a sequence such that
for all $n$, $x_n$ is $\varepsilon$-thick with respect to $g_n$. Then the sequence of pointed manifolds  $(M,g_n,x_n)$ has a subsequence that
converges in the $\mathcal C^2$ topology towards a hyperbolic pointed manifold with volume strictly less than $V_0(M)$.
\item[(3)]  The sequence $g_n$ has locally controlled curvature in the sense of Perelman.
\end{itemize}
Then $M$ is a graph manifold or contains an incompressible torus.
\end{theo}

Recall that if $M$ is a graph manifold, then $M$ is Seifert fibred or contains an incompressible
torus. Hence the conclusion of Theorem~\ref{TH:GRAPH} implies that $M$ satisfies the conclusion
of the Geometrisation Conjecture as stated at the beginning of this paper.

Note that Hypothesis~(2) of Theorem~\ref{TH:GRAPH} may be vacuous; this is
in particular the case if there is a sequence $\epsi_n\to 0$ such that for all $n$,  every point of $M$ is $\epsi_n$-thin with respect to $g_n$. In this situation, we shall say that the sequence $g_n$  \bydef{collapses.} Thus
 Theorem~\ref{TH:GRAPH} can be viewed as a \emph{weak collapsing} result in the sense that  we allow
the thick part to be nonempty, but require a control on its volume.
In the special case where $g_n$ collapses, the proof of  Theorem~\ref{TH:GRAPH} does not  use Hypothesis~(1), and leads to the conclusion that $M$ is a graph manifold.
Hence we also obtain the following version of Perelman's collapsing theorem:

\begin{corol}\label{cor:collapse} 
Let $M$ be a closed, orientable, irreducible, non-simply connected $3$-manifold. 
If $M$ admits a sequence of Riemannian metrics that collapses and has
controlled curvature in the sense of Perelman, then $M$ is a graph manifold.
\end{corol}

Sequences of metrics satisfying the hypotheses of Theorem~\ref{TH:GRAPH} are provided
by Perelman's construction and study of the Ricci flow with surgery~\cite{PerelmanII,PerelmanIII}. From these deep results and Theorem~\ref{TH:GRAPH} we deduce characterisations of hyperbolic and graph 3-manifolds: these are Theorems~\ref{th:minoration} and~\ref{corol:geom}. For this
we shall use the invariant $\overline{R}(M)$ defined below, which was  first suggested by M.~Anderson~\cite{Anderson2} (cf.\ also~\cite[\S 8]{PerelmanII} and~\cite[\S 93]{KleinerLott}.)

Let $M$ be a closed $3$-manifold. If $g$ is a Riemannian metric on $M$, we denote by
$R_{\textrm{min}}(g)$ the minimum of the scalar curvature $R$ of $g$ on $M$, and by
$\hat{R}(g)$ the scale invariant quantity $R_{\textrm{min}}(g) \vol(g)^{2/3}$. Note that
if $M$ has a hyperbolic metric $g_{hyp}$, then $\hat{R}(g_{hyp})$
is equal to $-6 \vol(g_{hyp})^{2/3}$.

We let $\overline{R}(M)$ be the (possibly infinite) supremum of $\hat R(g)$, taken over all Riemannian metrics $g$ on $M$.\footnote{When $M$ is aspherical, it turns out that
$\overline{R}(M)$ is finite, equal to the Yamabe invariant of $M$ (see e.g.~ \cite[Section 93]{KleinerLott}), but we will not use this fact.}

\begin{theo}\label{th:minoration} Let $M$ be a  closed, orientable, irreducible $3$-manifold. 
\begin{itemize}
\item[(1)] If $\overline{R}(M) \leq -6 V_0(M) ^{2/3}$, then $M$ is
hy\-per\-bo\-lic and $\overline{R}(M)=\hat{R}(g_{hyp})$.
\item[(2)] If $-6 V_0(M) ^{2/3} < \overline{R}(M)$, then $M$  contains an incompressible torus or is Seifert fibred.
\end{itemize}
\end{theo}

Theorem~\ref{th:minoration} immediately implies the Geometrization Conjecture.
Moreover, it shows that if $M$ is a closed hyperbolic $3$-manifold, then
the hyperbolic metric realises the maximum in $\overline{R}(M)$,
hence has the smallest volume among all complete metrics with scalar curvature bounded below by  $-6$. See~\cite{ADST} for an application.

Assertion~(1) of Theorem~\ref{th:minoration}  follows from the proof of Theorem~2.9
of~\cite{Anderson1}.  We shall give another proof based on Thurston's hyperbolic Dehn filling theorem, following~\cite[Chap. 3]{Salgueiro}. When $M$ is aspherical, Assertion~(2) is proved using Theorem~\ref{TH:GRAPH} and
results of Perelman on the long time behaviour of Ricci flow with surgery~\cite{PerelmanII}. (In fact, it is enough to assume that
$M$ is not simply connected.) For completeness, we have included in the
statement the case where $\pi_1M$ is finite, which follows from~\cite{PerelmanIII,
cm:question,MorganTian}.

Our last result is a complement to Theorem~\ref{th:minoration}. 
Let $V'_0(M)$ be the minimum of the volumes of all hyperbolic submanifolds $H\subset M$  having the property that $H$ is a link complement or $\bord \bar H$ has at
least one incompressible component.

\begin{theo}\label{corol:geom}
Let $M$ be a closed, orientable, irreducible $3$-manifold.
Then $M$ is a graph manifold if and only if $\overline{R}(M)>-6 V'_0(H)^{2/3}$.
\end{theo}

This article is organized as follows: in Section~1, we describe the strategy of
the proof of Theorem~\ref{TH:GRAPH}. This proof is then given in Sections~2--4.
In Section~5, we recall the necessary results coming from
the Ricci flow. In Section~6, we prove Theorems~\ref{th:minoration}
and~\ref{corol:geom}.

\paragraph{Acknowledgments}
The authors wish to thank John Morgan for  pointing out the proofs of 
Proposition~\ref{prop:nerf} and Corollary~\ref{cor:simplyconnected} which allowed to extend Theorem~\ref{TH:GRAPH} from the aspherical case to the non-simply connected case.

We also wish to thank the Clay Mathematics Institute 
for its financial support, as well as the project FEDER-MEC MTM2006-04353.

\section {Sketch of proof of Theorem~\ref{TH:GRAPH}}\label{sec:pres}
For classical $3$-manifold theory, we use \cite{Jaco}, \cite{Hempel} as main references, as well as ~\cite{BMP} for post-Thurston
results. To avoid any confusion between metric balls and topological balls, we shall call \bydef{$3$-ball} a $3$-manifold homeomorphic to the closed unit ball in $\Rr^3$. 
By contrast, our metric balls $B(x,r)$ are open.

Throughout the paper we work in the smooth category.
Recall that a \bydef{Haken manifold} is a connected, compact, orientable, irreducible $3$-manifold which contains an incompressible surface. Any connected, compact, orientable, irreducible $3$-manifold whose boundary is not empty is Haken. It follows from deep work of W.~Thurston
and earlier work of Jaco-Shalen and Johannson that every Haken manifold has a canonical
decomposition along incompressible tori into Seifert and hyperbolic pieces (see e.g.~the references given in~\cite{BMP}.) We call this
the \bydef{geometric decomposition} of the Haken manifold $M$.
Moreover, a Haken manifold is a  graph manifold if and only if all pieces in its geometric decomposition are Seifert.

Another key notion used in the proof of Theorem~\ref{TH:GRAPH} is
the \emph{simplicial volume,} sometimes called Gromov's norm, introduced by M.~Gromov in~\cite{GromovIHES}. Our proof relies on an additivity result for the simplicial volume under gluing along tori (see~\cite{GromovIHES,Kuessner,Soma}) which implies that  the simplicial volume of a $3$-manifold admitting a geometric decomposition is proportional to the sum of the volumes of the hyperbolic pieces.
In particular, such a manifold has zero simplicial volume if and only if it is a graph manifold.

We also use in an essential way Gromov's vanishing theorem \cite{GromovIHES,Ivanov}: if a $n$-dimensional closed manifold $M$ can be covered by open sets  $U_i$ such that the covering has dimension less than $n$ and the image of the canonical homomorphism $\pi_1 U_i\to \pi_1 M$ is amenable for all $i$, then the simplicial volume of $M$ vanishes.
(Recall that the dimension of a finite covering $U_i$ is the dimension of its nerve.)

Below we outline the proof of Theorem~\ref{TH:GRAPH}. For simplicity, we explain it in the special case where the sequence $g_n$ collapses.
In the last paragraph, we shall say a few words about the general case.

Before discussing the proof proper, we give an example of a \emph{covering argument} which
can be used to deduce topological information on $M$ (namely that $M$ has zero
simplicial volume) from the collapsing hypothesis.

For $n$ large enough, thanks to the local control on the curvature, each point has a neighbourhood in  $(M,g_n)$ which is close to a metric ball in some manifold of nonnegative sectional curvature, and whose volume is small compared to the cube of the radius. These neighbourhoods will be called \emph{local models}. From the classification of manifolds with nonnegative sectional curvature, we deduce that these local models have virtually abelian,  hence amenable fundamental groups.
A technique introduced by Gromov~\cite{GromovIHES} yields a covering of $M$, whose dimension is at most $2$, by open sets contained in these neighbourhoods. As a consequence, Gromov's vanishing theorem implies that the simplicial volume of $M$ vanishes.

The previous scheme, together with the additivity of the simplicial volume under gluing along incompressible tori, shows that a manifold which admits a geometric decomposition and a sequence of collapsing metrics is a graph manifold. This is however insufficient to prove Theorem~\ref{TH:GRAPH} since we do not assume that $M$ admits a geometric decomposition! Hence we need a trick
similar to those of~\cite{bp} and~\cite{blp}, which we now explain.

In the first step, we find a local model $U$ such that all connected components of $M\setminus U$ are
 Haken. This requirement is equivalent to irreducibility of each component of $M\setminus U$. Since $M$ is irreducible, it suffices to show that $U$ is not contained in a $3$-ball. This is in particular the case if $U$ is \bydef{homotopically nontrivial}, i.e.~the homomorphism $\pi_1(U)\to\pi_1(M)$ has
nontrivial image.

The proof of the existence of a homotopically nontrivial local model  $U$ is done by contradiction:
assuming that all local models are homotopically trivial,  we construct a covering of $M$ of dimension less than or equal to $2$ by homotopically trivial open sets. By a result of J.C. G\'omez-Larra\~naga and F. Gonz\'alez-Acu\~na \cite{GG}, a closed irreducible 3-manifold admitting such a covering must have trivial fundamental group. This is where we use the hypothesis that $M$ is not simply connected.

The second step, which is again a covering argument but done relatively to some fixed homotopically
nontrivial local model $U$, shows that any manifold obtained by 
Dehn filling on $Y:=M\setminus U$ has a covering of dimension 
less than or equal to $2$ by virtually abelian open 
sets, and therefore has vanishing simplicial volume.
We conclude using Proposition~\ref{prop:haken}, which states that
if $Y$ a  Haken manifold with boundary a collection of tori and such that the simplicial volume of every Dehn filling on $Y$ vanishes, then $Y$  is a graph manifold.\footnote{As already mentioned in~\cite{bp,blp}, Proposition~\ref{prop:haken} is a consequence of the geo\-me\-tri\-sa\-tion of Haken manifolds, additivity of the simplicial volume mentioned above, and Thurston's hyperbolic Dehn filling theorem.}
This finishes the sketch of proof of Theorem~\ref{TH:GRAPH} in the collapsing case.

In this text, we shall not separate the case when $g_n$ collapses, but we shall treat directly the general case. This implies that we first need to cover the thick part by submanifolds $H_n^i$ approximating compact cores of  limiting hyperbolic manifolds (Section~\ref{sec:epaisse}). We then cover the thin part by local models (Section~\ref{sec:mince}). The bulk of the proof is in Section~\ref{sec:preuve}: assuming
that $M$ contains no incompressible tori, we consider the covering of $M$ by approximately hyperbolic submanifolds and local models
of the thin part and perform two covering arguments: the first one shows that at least one of these
open subsets is homotopically nontrivial in $M$; the second one is done relatively to this homotopically
nontrivial subset and proves that $M$ is a graph manifold.

\section{Structure of the thick part}\label{sec:epaisse}

Until Section~\ref{sec:preuve}, we consider a $3$-manifold $M$ and a sequence of Riemannian metrics $g_n$ satisfying the hypotheses of Theorem~\ref{TH:GRAPH}. For the sake of simplicity, in the sequel we use the notation $M_n:= (M,g_n)$. The goal of this section is to describe the thick part of the manifolds $M_n$ and to make the link between the topology of the thick part and the topology of $M$.
We denote by $M_n^-(\varepsilon )$ the $\varepsilon$-thin part of $M_n$, and by
$M_n^+(\varepsilon)$ its $\varepsilon$-thick part. 
\subsection{Covering the thick part}

\begin{prop}\label{partie epaisse}
Up to taking a subsequence of $M_n$, 
there exists a finite (possibly empty) collection  of pointed hyperbolic manifolds $(H^1,*^1),
\ldots,$ $ (H^m,*^m)$ and for every $1 \leq i\leq m$ a sequence $x_n^i\in M_n$ satisfying:
\begin{itemize}
\item[(i)] $\lim\limits_{n\to\infty}(M_n,x_n^i)=(H^i,*^i)$ in the $\mathcal C^2$ topology.
\item[(ii)] For all sufficiently small $\varepsilon>0$, there exist $n_0(\varepsilon)$ and $C(\varepsilon)$ such that for all $n\geq n_0(\varepsilon)$ one has $M_n^+(\varepsilon)\subset \bigcup_i B(x_n^i,C(\varepsilon))$.
\end{itemize}
\end{prop}

\begin{proof}

By assumption, the sequence  $\vol(M_n)$ is bounded above.
Let $\mu_0>0$ be a universal number such that any hyperbolic manifold has volume at least  $\mu_0$.

If for all $\varepsilon>0$ we have $M_n^+(\varepsilon)=\emptyset$ for $n$ large enough, then
Proposition~\ref{partie epaisse} is vacuously true. Otherwise, we use Hypothesis (2) of Theorem~\ref{TH:GRAPH}: up to taking a subsequence of $M_n$, there exists $\varepsilon_1>0$ and a sequence of points $x^1_n\in M_n^+(\varepsilon_1)$ such that $(M_n,x^1_n)$ converges to a pointed hyperbolic manifold $(H^1, *^1)$.

If for all $\varepsilon>0$ there exists $C(\epsilon)$ such that, for $n$ large enough, $M_n^+(\varepsilon)$ is included in $B(x^1_n,C(\varepsilon))$, then we are done. Otherwise there exists $\varepsilon_2>0$ and a sequence $x^2_n\in M_n^+(\varepsilon_2)$ such that  $d(x^1_n,x^2_n)\to\infty$. Again Hypothesis~(2) 
of Theorem~\ref{TH:GRAPH} ensures that, after taking a subsequence, the sequence $(M_n,x^2_n)$ converges to a pointed hyperbolic manifold $(H^2, *^2)$.

Note that for each $i$, and for $n$ sufficiently large, $M_n$ contains a submanifold  $\mathcal C^2$-close to some compact core of $H_i$ and whose volume is greater than or equal to $\mu_0/2$. Moreover, for $n$ fixed and large, these submanifolds are pairwise disjoint.
Since the volume of the manifolds $M^n$ is uniformly bounded above this construction has to stop. Condition~(ii) of the conclusion of Proposition~\ref{partie epaisse} is then satisfied for  $0<\varepsilon<\varepsilon_k$.
\end{proof}

\begin{rem}
By Proposition~\ref{partie epaisse} one can choose sequences $\varepsilon_n\to 0$ and $r_n\to \infty$ 
such that the  ball $B(x^i_n,r_n)$ is arbitrarily close to a metric ball  $B(*^i,r_n)\subset H^i$, 
for $i=1,\ldots, m $, and every point of $M_n\setminus \bigcup_i B(x^i_n,r_n)$ is $\varepsilon_n$-thin.
\end{rem}

Let us fix a sequence of positive real numbers $\varepsilon_n\to 0$.
Let $H^1,\ldots,$ $H^m$ be hyperbolic limits given by 
Proposition~\ref{partie epaisse}.  For each $i$ we choose a compact core $\bar H^i$
for $H^i$ and for each $n$ a submanifold $\bar H_n^i$ and an
approximation $\phi_n^i:\bar H_n^i \to \bar H^i$. Up to renumbering, one can assume that for all $n$ 
{we have $M_n\setminus\bigcup \bar H^i_n\subset
M_n^-(\varepsilon_n)$}, and that the $\bar H^i_n$'s are disjoint.

The hypothesis that the volume of each hyperbolic limit $H^i$ is less than $V_0$ implies that for $n$ sufficiently large no component  $\bar H^i_n$ is homeomorphic to the exterior of a link in $M$.

The logic of the proof is the following:
each boundary component of $\bar H^i_n$  is a torus. If one of those tori is incompressible, then the conclusion of Theorem~\ref{TH:GRAPH} is true. The interesting case is when all the tori that appear in the boundary of the thick part are compressible. The remainder of this section is devoted to the
two following results:

\begin{prop}\label{prop:abelien} 
Up to taking a subsequence, one of the following properties is satisfied:
\begin{itemize}
\item[(i)] There exists an integer $i_0 \in \{1,\ldots, m \}$ such that
$\partial \bar H_n^{i_0}$ contains an incompressible torus for all $n$, or
\item[(ii)]  for all $i \in \{1,\ldots, m \}$,  
$\bar H_n^{i}$ is embedded in a solid torus or in a $3$-ball contained in $M_n$ for all $n$.
\end{itemize}
\end{prop}

\begin{prop}\label{prop:existence Wn}
If Conclusion (ii) of Proposition~\ref{prop:abelien}
is satisfied, then either $M$ is a lens space or there exists for each $n$ a submanifold $W_n\subset M_n$
such that:
\begin{itemize}
\item[(i)] $\bigcup \bar H_n^i  \subset W_n$.
\item[(ii)] Each connected component of  $W_n$ is a solid torus, or contained
in a $3$-ball and homeomorphic to the exterior of a knot in $S^3$.
\item[(iii)] The boundary of each component of $W_n$ is a component of $\bigcup_i \partial \bar H_n^i$.
\end{itemize}
\end{prop}

Subsection~\ref{un peu de topo} is devoted to general topological results concerning compressible tori in $3$-manifolds and abelian submanifolds. Proposition~\ref{prop:abelien} and~\ref{prop:existence Wn} will be proved in the Subsection~\ref{preuve abelien Wn}.

\subsection{Submanifolds with compressible boundary}
\label{un peu de topo}

Let $X$ be an orientable, irreducible $3$-manifold and $T$ be a compressible torus embedded in  $X$.
The Loop Theorem shows the existence of a \bydef{compression disc} $D$ for $T$, that is, a
disc $D$ embedded in $M$ such that $D\cap T=\bord D$ and the curve
$\bord D$ is not null homotopic in  $T$. By cutting open $T$ along an open small regular neighbourhood of $D$ and gluing two parallel copies of $D$ along the boundary curves,  one  constructs an embedded 2-sphere
 $S$ in $X$. We say that $S$ is obtained by \bydef{compressing $T$ along $D$}.

Since $X$ is assumed to be irreducible, $S$ bounds a $3$-ball $B$. There are two possible situations depending on whether $B$ contains $T$ or not. The following lemma collects some standard results that we shall need.

\begin{lem}\label{lem:tore compressible}
Let $X$ be an orientable, irreducible $3$-manifold and $T$ be a compressible torus embedded in $X$. Let $D$ be a compression disc for $T$, $S$ be a sphere obtained by compressing $T$ along $D$, and $B$
a ball bounded by $S$. Then:
\begin{itemize}
\item[i)] $X\setminus T$ has two connected components $U,V$, and $D$
is contained in the closure of one of them, say $U$.
\item[ii)] If $B$ does not contain $T$, then $B$ is contained in $\bar U$, and
$\bar U$ is a solid torus.
\item[iii)] If $B$ contains $T$, then $B$ contains $V$, and $\bar V$ is
homeomorphic to the exterior of a knot in $S^3$. In this case, there exists a homeomorphism $f$
from the boundary of $S^1\times D^2$ into $T$ such that the manifold obtained by gluing  $S^1\times D^2$
to $\bar U$ along $f$ is homeomorphic to $X$.
\end{itemize}
\end{lem}

\begin{rem}
If $T$ is a component of $\bord X$ and  $T$ is a compressible torus,  the same argument shows that  $X$ is a solid torus.
\end{rem}

\begin{lem}\label{lem:abelien}
Let $X$ be a closed, orientable, irreducible $3$-manifold. Let $\bar H \subset X$ be a connected, compact, orientable, irreducible submanifold of $X$ whose boundary is a collection of compressible tori. 
If $\bar X$ is not homeomorphic to the exterior of a (possibly empty) link in $X$, then $\bar H$ is included in a connected submanifold $Y$ whose boundary is one of the  tori of $\partial \bar H$
and which satisfies one of the following properties:
\begin{itemize}
\item[(i)] $Y$ is a solid torus, or
\item[(ii)] $Y$ is homeomorphic to the exterior of a knot in 
$S^3$ and contained in a ball $B\subset X$.
\end{itemize}
\end{lem}
 
\begin{proof}
By hypothesis the boundary of $\bar H$  is not empty.
We denote by $T_1,\ldots,T_m$ the components of $\partial \bar H$. If one of them bounds a solid torus containing $\bar H$, we can choose this solid torus as $Y$. 
Henceforth we assume that this is not the case.

Each $T_j$ being compressible, it separates and thus bounds a submanifold $V_j$ not containing
$\bar H$. Up to renumbering the boundary components of $\bar H$, we may assume that
 $V_1,\ldots,V_k$ are solid tori, but not $V_{k+1},\ldots, V_m$.
At least one of the $V_j$'s is not a solid torus, otherwise $\bar H$ would be homeomorphic to the exterior of a link in $X$. 

For the same reason,  at least one 
 $V_j$, for some $j > k$, is not contained in a $3$-ball. Otherwise each of the $V_{k+1},\ldots, V_m$
 is homeomorphic to the exterior of a knot in  $S^3$, by Lemma~\ref{lem:tore compressible}, and one could then replace each
 $V_j$ ,$k+1 \leq j \leq m$ by a solid torus without changing the topological type of  $X$. Hence $\bar H$ would be homeomorphic to the exterior of a link in $X$.

Pick a  $V_j$, for $j > k$,  which is not contained in a ball. Then compressing surgery on the torus
 $T_j = \partial V_j$ yields a sphere $S$ bounding a ball $B$ in $X$, which contains $\bar H$ by the choice of
$V_j$.  This shows that conclusion (ii) is satisfied with
$Y=X\setminus \mathrm{int} V_j$.
\end{proof}

\subsection{Proof of Propositions~\ref{prop:abelien}
and~\ref{prop:existence Wn}}\label{preuve abelien Wn}

\begin{proof}[Proof of~\ref{prop:abelien}]
If Assertion (i) of Proposition~\ref{prop:abelien} is not satisfied, then, up to a subsequence, one may
assume that for all $i \in{1,\ldots, m}$ 
and for all $n$, each component of
$\partial \bar H_n^{i}$ is compressible in $M$.  We fix an integer
$i \in{1,\ldots, m}$. 
From the hypotheses of Theorem~\ref{TH:GRAPH}, we get the inequality
 $\vol(H^i) < V_0(M)$, which implies that $H^i$ is not homeomorphic 
to the complement of a link in $M$. In particular since  $\bar H_n^{i}$ 
is homeomorphic to the compact core of $H^i$, it is not homeomorphic to 
the exterior of a link in
$M$. Lemma~\ref{lem:abelien} allows to conclude that Assertion~(ii) of Proposition~\ref{prop:abelien} holds true.
\end{proof}

\begin{proof}[Proof of~\ref{prop:existence Wn}]
For each $n$ and each $i \in{1,\ldots, m}$, we choose a submanifold $Y^i_n$
containing $\bar H^i_n$,  given by Lemma~\ref{lem:abelien}.
We take $W_n$ to be the union of the $Y^i_n$. Then Assertion~(i) of
Proposition~\ref{prop:existence Wn} is straightforward.

Assume that $M$ is not a lens space. Then $M_n$ cannot be 
the union of two submanifolds  $Y^i_n$  and $Y^j_n$, 
otherwise $M_n$ can be covered either by two solid tori, 
or by a solid torus and a ball or by two balls. In the first case
$M_n$ would be homeomorphic to a lens space by \cite{GGH}, 
while in the other two cases $M_n$ would be covered by three 
balls and thus homeomorphic to the 3-sphere $S^3$ by \cite{HMc}, 
see also \cite{GGH1}. Thus for all $i_1,i_2$, the submanifolds $Y^{i_1}_n$ $Y^{i_2}_n$ are disjoint 
or one contains the other, because they  have disjoint boundaries. 
In this case each component of $W_n$ is homeomorphic to one of the $Y^i_n$. 
This yields Assertions~(ii) and~(iii) of Proposition~\ref{prop:existence Wn}.
\end{proof}

\section{Local structure of the thin part}\label{sec:mince}

In this section, it is implicit that any quantity depending on a point $x \in M_n$ is computed with
respect to the metric $g_n$ on $M_n$ and thus depends also on  $n$.

Let us choose a sequence $\varepsilon_n\to 0$ (see the remark after Proposition~\ref{partie epaisse}).
For all $x\in M_n^-(\varepsilon_n)$, we choose a radius $0<\rho(x)\leq
1$, such that on the ball $B(x,\rho(x))$ the curvature is $\geq -\rho^{-2}(x)$
and the volume of this ball is $< \varepsilon_n \rho^3(x)$.

In the following proposition we use Cheeger-Gromoll's soul theorem  \cite{CheegerGromoll}.

\begin{prop}\label{local structure}	
For all $D>1$ there exists $n_0(D)$ such that if $n> n_0(D)$, then for all
$x\in M_n^-(\varepsilon_n)$ we have the following alternative:
\begin{itemize}
\item[(a)]  Either $M_n$ is $\frac1D$-close to some closed nonnegatively curved $3$-manifold, or
\item[(b)]  there exists a radius $\nu(x)\in (0,\rho(x))$ and a complete noncompact Riemannian $3$-manifold
 $X_x$, with nonnegative sectional curvature and soul  $S_x$, such that the following properties are satisfied:
\begin{itemize}
\item[(1)] $B(x,\nu(x))$ is $\frac1D$-close to a metric ball in $X_x$.
\item[(2)] There exists an approximation $f_{x}\! : \! B(x,\nu(x))\to X_x$ such that
$$\max\{d(f(x),S_x), \diam S_x\}\leq \tfrac{\nu(x)}D.$$
\item[(3)] $\vol (B(x,\nu(x)))\leq \frac 1D\nu^3(x)$.
\end{itemize}
\end{itemize}
\end{prop}

\begin{rem}
Since $\nu(x)<\rho(x)$, the sectional curvature on $B(x,\nu(x))$ is greater than or equal to
$ -\frac 1{\rho^2(x)}$, which is in turn bounded below by $-\frac 1{\nu^2(x)}$.
\end{rem}

\begin{rem}
The only closed, orientable and irreducible $3$-manifold containing a projective plane is  $RP^3$, which is a graph manifold. Therefore if the manifold $M$ is not homeomorphic to $RP^3$, then the soul $S_x$ can be homeomorphic to a point, a circle, a
$2$-sphere, a $2$-torus or a Klein bottle.  In this case, the ball
$B(x,\nu(x))$ is homeomorphic to $B^3$, $S^1\times D^2$,
$S^2\times I$, $T^2\times I$ or to the twisted $I$-bundle on the Klein bottle.  
\end{rem}

Before starting the proof of this proposition, we prove the following lemma and its consequence:

\begin{lem}\label{lem:truc}
There exists a universal constant $C>0$ such that for all
$\varepsilon > 0$, for all $x\in M_n$, and for all $r>0$, if the ball $B(x,r)$ has volume $\geq \varepsilon\, r^3$ and curvature $\geq
-r^{-2}$, then for all $y\in B(x, \frac13 r)$ and all $0<r'<\frac23
r$, the ball $B(y,r')$ has volume $\geq C\cdot \varepsilon (r')^3$ and curvature  
{$\geq -(r')^{ -2}$}.
\end{lem}

We use the function $v_{-\kappa^2}(r)$ to denote the volume of the ball of radius $r$ in the $3$-dimensional hyperbolic space with curvature $-\kappa^2$. Notice that $v_{-\kappa^2}(r)=\kappa^{-3} v_{-1}(\kappa\, r)$.

\begin{proof}
The lower bound on the curvature is a consequence of the monotonicity of the function $-r^{-2}$ with respect to $r$. In order to estimate from below the normalised volume we apply Bishop-Gromov's inequality twice. First  to the ball around $y$, increasing the radius $r'$ to $\frac 23 r$: 
$$
\vol (B(y, r'))\geq \vol ( B(y,\tfrac23 r)
)\frac{v_{-r^{-2}}(r')}{v_{-r^{-2}}(\frac23r)}.
$$
Using that 
$ {v_{-r^{-2}}(r')}= r^3 {v_{-1}(\frac{r'}{r})}\geq 
 r^3 \left(\tfrac{r'}{r}\right)^3 C_1
$ for $C_1>0$ uniform, 
${v_{-r^{-2}}(\frac23r)} = r^3 {v_{-1}(\frac23)}$,
 and
that the ball $B(y,\frac23 r)$
contains $B(x,\frac13 r)$, we have
$$
\vol (B(y, r'))\geq \vol( B(x,\tfrac13 r)) \left(\tfrac{r'}{r}\right)^3 C_2.
$$
Applying again the Bishop-Gromov inequality: 
$$
\vol (B(x,\tfrac13 r)) \geq \vol (B(x,r)) \frac{v_{-r^{-2}}(\frac13 r)}{v_{-r^{-2}}(r)}\geq
r^3 \varepsilon\, \frac{v_{-1}(\frac 13)}{v_{-1}(1)}= r^3\varepsilon\, C_3.
$$
Hence  $\vol (B(y, r'))\geq (r' )^3\, \varepsilon \, C_4$.
\end{proof}

We deduce an `improvement' of the controlled curvature in the sense of Perelman, in which the conclusion is valid at each point of some metric ball, not only the centre. The only price
to pay is that the constants can be  slightly different.

\begin{corol}\label{corol:bornecourbure}
For all $\varepsilon>0$ there exists $\bar r'(\varepsilon)>0$, $K'_0(\varepsilon),
K'_1(\varepsilon)$ such that for  $n$ large enough, if $0 < r \leq \bar r'
(\varepsilon)$,  $x\in M_n$ and the ball $B(x,r)$ has volume $\geq \varepsilon \,
r^3$ and
sectional curvatures $\geq - r^{-2}$ then, for all $y\in B(x,\frac13
r)$, 
$\vert \Rm(y)\vert < K'_0 \, r^{-2}$ and $\vert \nabla \Rm(y)\vert < K'_1\, r^{-3}$.
\end{corol}

\begin{proof}
It suffices to apply Lemma~\ref{lem:truc}, setting $\bar r'(\varepsilon) = \bar r(C\epsilon)$,
$K'_0(\varepsilon) = K_0(C\epsilon)$ and $K'_1(\varepsilon) = K_1(C\epsilon)$.
\end{proof}

\begin{proof}[Proof of Proposition \ref{local structure}.]
Let us assume that there exists $D_0>1$ and, after re-indexing, a sequence $x_n\in M_n^-(\varepsilon_n)$ such that neither of the
conclusions of Proposition \ref{local structure} holds with $D=D_0$.
 
Set  $\varepsilon_0:=\frac1{1000 D_0}$. We shall rescale the metrics using the following radii:
\begin{defi} For $x\in M_n$, define
\[
\er(x)=\inf\{ r>0\mid\vol ( B(x,r) )/r^3\leq \,\varepsilon_0 \}.
\]
\end{defi}

We gather in the following lemma some properties which will be useful for the proof:

\begin{lem} 
\label{aff:er}
 \begin{itemize}
\item[(i)] For  $n$ large enough and $x\in M_n^-(\varepsilon_n)$, one has $0<\er(x)<\rho(x)$.  
\item[(ii)] For $n$ sufficiently large and  $x\in M_n^-(\varepsilon_n)$, one has
 $$\frac{\vol(B(x,\er(x)))}{\er(x)^3}=\varepsilon_0.$$
\item[(iii)] For  $L>1$, there exists $n_0(L)$ such that for
$n>n_0(L)$ and for   $x\in M_n^-(\varepsilon_n)$ we have $$L \,\er(x)\leq \rho(x). 
$$
In particular $\lim\limits_{n\to\infty} \er(x_n)= 
\lim\limits_{n\to\infty}\frac{ \er(x_n)}{\rho(x_n)}=0$.
\end{itemize}
\end{lem}

\begin{proof}[Proof of Lemma~\ref{aff:er}]
Property (i) follows from continuity, by comparing the limit
 $\vol(B(x,\delta))/\delta^3\to \frac43\pi$ when $\delta\to 0$
with $$\vol(B(x,\rho(x)))/\rho(x)^3 < \varepsilon_n\to 0.$$
Assertion (ii) is also proved by continuity.

We prove (iii)  for $L > 1$ using the function
$$
f_x(s)=\frac{\vol(B(x,s\er(x)))}{(s\er(x))^3}.
$$
 One has
$f_x(1)=\varepsilon_0$; for all $s\in[1, L]$, $f_x(s)\geq \frac{\varepsilon_0}{s^3}
\geq 
\frac{\varepsilon_0}{L^3}$. Furthermore, for $s < 1$, $f_x(s) >\varepsilon_0$ by the definition of  $\er(x)$. It suffices then to choose 
$n_0$ so that for all  $n\geq n_0$ one has
$\varepsilon_n < \frac{\varepsilon_0}{L^3}$.
\end{proof}

\medskip

\begin{rem}
 For $n$ large enough, from the preceding Lemma,
we have $\er(x_n)<\bar r'(\varepsilon_0)$.
\end{rem}

\begin{corol}\label{corol:util} There exists a constant $C>0$ such that any sequence of $x_n\in M_n$ satisfies
$$
\frac{\inj(x_n)}{\er(x_n)}\geq C
$$
for $n$ large enough.
\end{corol} 

\begin{proof}
 Let us first remark that, since $\er(x_n)<\rho(x_n)$, the sectional curvatures on $B(x_n,\er(x_n))$ are $\geq
-\frac1{\rho(x_n)^2}>-\frac{1}{\er(x_n)^2}$. Moreover, as $\er(x_n)<
\bar r'(\varepsilon_0)$, Corollary~\ref{corol:bornecourbure}
shows that the curvature on the ball $B(x_n,\frac{\er(x_n)}3)$ is bounded above by $K'_0(\varepsilon_0)/\er(x_n)^2$. 
This rescaled ball
$$
\frac{1}{\er(x_n)}B(x_n,\tfrac13 \er (x_n))
$$
has radius $\leq 1$, volume $\geq \varepsilon_0\, C_0$ 
(where $C_0$ is a universal constant coming from Bishop-Gromov) and curvatures
$\leq K'_0(\varepsilon_0)$.
Using Cheeger's propeller lemma~\cite[Thm.\  5.8]{CheegerEbin}, the injectivity radius at the centre of the rescaled ball 
is bounded below by some constant $C>0$. This proves Corollary~\ref{corol:util}.
\end{proof}

Having proved Lemma~\ref{aff:er} and its corollary, we continue the proof of 
Proposition~\ref{local structure}.
Let us consider the rescaled manifold $\overline
 M_n=\frac1{\er(x_n)} M_n$. We look for a limit of the sequence
 $(\overline M_n,\bar x_n)$, where $\bar x_n$ is the image of $x_n$.  
The ball  $B(\bar x_n,\frac{\rho(x_n)}{\er(x_n)})\subset \overline M_n$ 
has sectional curvature bounded below by $-\left(\frac{\er(x_n)}{\rho(x_n)}\right)^2$,
which goes to $0$ when
 $n\to\infty$, as follows from Assertion~(iii) of Lemma~\ref{aff:er}.
 
Given $L>0$, the ball $B(\bar x_n, 3L)$ is obtained by rescaling the ball $B(x_n, 3 L\er(x_n))$. Since 
$3L\er(x_n) < \rho(x_n)$, 
 the sectional curvature on
$B(x_n,3 L\er(x_n))$ is $\geq -\frac{1}{\rho(x_n)^2}\geq -\frac{1}{(3L\er(x_n))^2}$. Moreover, we have
 
$$\frac{\vol( B(x_n,3 L\er(x_n))}{(3 L\er(x_n))^3} \geq \frac{\varepsilon_0}{(3L)^3}.$$ 

By applying Corollary ~\ref{corol:bornecourbure} for $n$ sufficiently large so that  we have 
$3 L\er(x_n)\leq   \bar r'(\frac{\varepsilon_0}{(3L)^3})$, one gets that the curvature 
is locally controlled in the sense of Perelman at each point of the ball $B(x_n, L\er(x_n))$. 
Therefore 
the curvature and its first derivative can be bounded above on any  
ball $B(\bar x_n, L)\subset \overline M_n$ with a given radius $L >0$. 

Since the injectivity radius of the basepoint $x_n$ is bounded below along the sequence, 
this upper bound on the curvature allows to use Gromov's compactness
theorem~\cite[Chap. 8, Thm.\  8.28]{GLP}, 
\cite{peters:compactness} and its versions with regularised limit \cite[Thm.\  2.3]{hamilton:compactness} or \cite[Thm.\  4.1 and  5.10]{fukaya:convergence}. It follows that the pointed sequence $(\overline M_n,\bar x_n)$ subconverges in the $\mathcal C^2$-topology towards a 3-dimensional smooth manifold $(\overline X_{\infty},x_{\infty})$, with a complete riemannian metric of class $\mathcal C^2$ with nonnegative sectional curvature. This limit manifold cannot be closed, because that would
contradict the assumption that the conclusion of Proposition~{\ref{local structure} does not hold.

Hence $\overline X_{\infty}$ is not compact. Let $\overline{S}$ be its soul.	Let us choose
$$
      \nu(x_n)= L\er(x_n)\qquad\textrm{where\ } L\geq 2\,\diam
   (\overline S\cup \{x_{\infty}\})\, D_0.
$$
for $n$ large  (to be specified later) we set 
$$
X_{x_{n}}=\er(x_n) \overline X_{\infty},
\quad\textrm{ and }\quad S_{x_{n}}=\er(x_n)\overline S.
$$ 
We then have
$$
\diam(S_{x_{n}})=\er(x_n)\diam (\overline{S}) < \nu(x_n)/D_0.
$$

Let $\bar f_n\!: B(\bar x_n,L )\to  (\overline X_{\infty},x_\infty)$ be a
$\delta_n$-approximation, where $\delta_n$ is a sequence going to $0$.
After rescaling $f_n\! : B(x_n, L \er(x_n) )\to X_{x_{n}}$ is
also a  $\delta_n$-approximation. We get:
\begin{multline*}
 d(f_n(x_n),S_{x_{n}})= \er(x_n)\, d(\bar f_n(\bar x_n),\overline S) 
\\
\le
\er(x_n)\big( d(\bar f_n(\bar x_n),\bar x_\infty)+ d(\bar
x_\infty,\overline S)\big) \\ \leq \er(x_n) \delta_n+
\frac{\nu(x_n)}{2 D_0}\leq \frac{\nu(x_n)}{D_0}.
\end{multline*}

This proves assertion~(2) of Proposition~\ref{local structure}.

Using the fact that $\nu(x_n)=L \er(x_n)<\rho(x_n)$, the curvature on $B(x_n,\nu(x_n))$ is $\geq -1/\nu(x_n)^2$, $L>1$ and the
Bishop-Gromov inequality, we get:
\begin{multline*}
\frac{\vol(B(x_n,\nu(x_n)))}{v_{-\frac{1}{\nu^2(x_n)}}(\nu(x_n))}
\leq
\frac{\vol(B(x_n,\er(x_n)))}{v_{-\frac{1}{\nu^2(x_n)}}(\er(x_n))}=
\varepsilon_0 \frac{\er(x_n)^3}{v_{-\frac{1}{\nu^2(x_n)}}(\er(x_n))}
= 
\\
\varepsilon_0 \left(\frac{\er(x_n)}{\nu(x_n)}\right)^3 \frac{1}{v_{-1}(\frac{\er(x_n)}{\nu(x_n)})}
= \varepsilon_0 \frac{1}{L^3} \frac{1}{v_{-1}(\frac{1}{L})}.
\end{multline*}
Taking now $L$ sufficiently large, we find that:
$$
\vol(B(x_n,\nu(x_n)))\leq \varepsilon_0 \frac{1}{L^3} \frac{v_{-1}(1)}{v_{-1}(\frac{1}{L})} \nu^3(x_n)
\leq 1000\,\varepsilon_0 \,\nu^3(x_n)= \frac{1}{D_0} \nu^3(x_n),
$$
where the last equality comes from the definition of $\varepsilon_0$.

Hence we get the contradiction required to conclude the proof of Proposition~\ref{local structure}.}
\end{proof}

\section{Constructions of  coverings}\label{sec:preuve}

\subsection{Embedding thick pieces in solid tori}

We begin by making some reductions for the proof of Theorem~\ref{TH:GRAPH}.

If case (a) of Proposition~\ref{local structure} occurs, then $M$ is a closed, orientable, irreducible
$3$-manifold admitting a metric of nonnegative
sectional curvature.	By
	\cite{Hamilton3,Hamilton4}, $M$ is spherical or Euclidean, hence a graph manifold.
	Therefore we may assume that all local models are noncompact.
	
	 For the same reasons, since lens spaces are graph manifolds,  we can also assume that $M$ is not homeomorphic to a lens space, and in particular does not contain a projective plane. 

If there exists an integer $i_0 \in{1,\ldots, m}$ such that, up to a subsequence,
$\partial \bar H_n^{i}$ contains an incompressible torus for all $n$, then Theorem~\ref{TH:GRAPH} is proved. We thus assume that for all
$i \in{1,\ldots, m}$ and for all $n$, each component of
$\partial \bar H_n^{i}$ is compressible in $M_n$ and thus  Propositions~\ref{prop:abelien} 
and~\ref{prop:existence Wn}  apply and give for each $n$ a submanifold $W_n$.

%In the sequel we shall need to construct abelian coverings of  $M$ relatively to a subset $\mathcal{V}\subset M$, 
%that is by open sets $U$ such that the
%image of $\pi_1(U\setminus (U\cap \mathcal{V})) \to
%\pi_1(M\setminus\mathcal{V})$ is (virtually) abelian.  Here we face a technical difficulty: although the components of $W_n$
%which are not solid tori are abelian in $M$, they may not remain such in the complement of a given subset.
%In order to circumvent this difficulty we replace each component of  $W_n$ by a solid torus without changing the topology of $M_n$, 
%but changing the metric as we describe next.

Assume that there exists a component $X$ of $W_n$ which is not a solid torus. 
From Proposition~\ref{prop:existence Wn}(ii), $X$ is a knot exterior and contained in a $3$-ball $B\subset M_n$.
By Lemma~\ref{lem:tore compressible}, it is possible to replace 
 $X$ by a solid torus $Y$ without changing the global topology.
Let us denote by  $M'_n$ the manifold thus obtained.  We can endow 
$M'_n$ with a Riemannian metric $g'_n$, equal to $g_n$ away from  $Y$ and such that 
an arbitrarily large  collar neighbourhood of $\bord Y$ in  $Y$ is isometric to a collar neighbourhood
$\bord X$ in $X$. When $n$ is large, this neighbourhood is thus almost isometric
to a long piece of a hyperbolic cusp, and this geometric property will be sufficient for our covering arguments.

Repeating this construction for each component of $W_n$ which is not a solid torus, we obtain a Riemannian manifold
$(M''_n,g''_n)$ together with a submanifold $W''_n$ satisfying the following properties:

\begin{itemize}
\item[(i)] $M''_n$ is homeomorphic to $M_n$.
\item[(ii)] $M''_n\setminus W''_n$ is equal to $M_n\setminus W_n$ and the metrics $g_n$
and $g''_n$ coincide on this set.
\item[(iii)] $M''_n\setminus W''_n=M_n\setminus W_n$ is $\varepsilon_n$-thin.
\item[(iv)] When $n$ goes to infinity, there exists a collar neighbourhood of $\bord W''_n$ in $W''_n$
of arbitrarily large diameter isometric to the corresponding neighbourhood in $W_n$.
\item[(v)] Each component of $W''_n$ is a solid torus.
\end{itemize}

For simplicity, we use the notation  $M_n$, $g_n$, $W_n$ instead of $M''_n$, $g''_n$, $W''_n$. 
This amounts to assuming in the conclusion of Proposition~\ref{prop:existence Wn} that all components of  $W_n$ are solid tori.

\subsection{Existence of a homotopically non\-tri\-vial o\-pen set}
We say that an arcwise connected set $U \subset M$ is
\bydef{homotopically trivial } (in $M$) 
if the image of the  homomorphism 
$\pi_1(U)\to \pi_1(M)$
is trivial. 
More generally, we say that the subset $U\subset M$
is homotopically trivial if all its arcwise connected components have this property.

We recall that the  \bydef{dimension} of a finite covering $\{ U_i \}_{i}$ of $M$ is the 
dimension of its nerve, hence the dimension plus one equals
the maximal number of $U_i$'s containing a given point.

\begin{prop} 
\label{prop:notrivialpi1}
There exists $D_0>0$ such that for all $D>D_0$, for every $n$ greater than or equal
to the number $n_0(D)$ given by Proposition~\ref{local structure}, one of the following
assertions is true:
\begin{itemize}
\item[(a)]  some connected component of $W_n$ is not homotopically trivial, or
\item[(b)]  there exists $x\in {M_n}\setminus \mathrm{int} (W_n)$ such that the image
of \break $\pi_1(B(x,\nu(x)))\to\pi_1(M_n)$
is not homotopically trivial.
\end{itemize}
\end{prop}

In \cite{GG}  J.C. G\'omez-Larra\~naga and F. Gonz\'alez-Acu\~na have computed the $1$-dimensional Lusternik-Schnirelmann category of a closed 3-manifold. One step of their proof  gives the following proposition (cf.~\cite[Proof of Prop.\ 2.1]{GG}:)

\begin{prop}\label{prop:nerf} Let $X$ be a closed, connected 3-manifold. If $X$ has a covering of dimension 2 by open subsets which are homotopically trivial in $X$, then there is a connected $2$-dimensional complex $K$ and a continuous map $f: X \to K$ such that the induced homomorphism $f_{\star}: \pi_1(X) \to \pi_1(K)$ is an isomorphism. \qed
\end{prop}

Standard homological arguments show the following, cf.\ \cite[\S 3]{GG}:

\begin{corol}\label{cor:simplyconnected} Let $X$ be a closed, connected, orientable, irreducible 3-manifold.  If $X$ has a covering of dimension 2 by open subsets which are homotopically trivial in $X$, then $X$ is simply connected.
\end{corol}

\begin{proof} By Proposition~\ref{prop:nerf}, 
let $f\!: X \to K$ be a continuous map from $X$ to a connected $2$-dimensional complex $K$, such that the 
induced homomorphism $f_{\star}\!: \pi_1(X) \to \pi_1(K)$ is an isomorphism. Let $Z$ be a $K(\pi_{1}(X),1)$ space.  Let $\phi\! : X \to Z$ be a map from $X$ 
to $Z$ realizing the identity homomorphism on $\pi_{1}(X)$ and let $\psi\!: K \to Z$  
be the map from $K$ to $Z$ realizing the isomorphism $f_{\star}^{-1}: \!
\pi_1(K) \to \pi_{1}(X)$.  Then $\phi$ is homotopic to  $\psi\circ f$ and  the induced homomorphism 
$\phi_{*}\!: H_3(X;\ZZ) \to H_3(Z;\ZZ)$ factors through $\psi_{*}\!: H_3(K;\ZZ) \to H_3(Z;\ZZ)$. 
Since $H_3(K;\ZZ) = \{0\}$, the homomorphism  
$\phi_{*}$ must be trivial. 
$$
\begin{array}{rcl}
 X & \!\!\! \overset \phi \longrightarrow & \!\!\! Z \\
 f\downarrow & \nearrow _\psi  & \\
K & & 
\end{array}
$$

If $\pi_1(X)$ is infinite, then $X$ is aspherical and $\phi_{*}$ is an isomorphism. 
Therefore $\pi_1(X)$ is finite.

If $\pi_1(X)$ is finite of order $d >1$, then  let $\widetilde X$ be the universal covering of $X$. 
The covering map $p: \widetilde X \to X$ induces an isomorphism between the homotopy groups 
$\pi_k(\widetilde X)$  and 
$\pi_k(X)$ for $k\geq 2$. Since $\pi_2(X) =\{0\}$, $\pi_2(\widetilde X) =\{0\}$, and by the Hurewicz theorem, the canonical homomorphism 
$\pi_3(\widetilde X) \to H_3(\widetilde X;\ZZ) = \ZZ$  is an isomorphism. It follows that the canonical map $\pi_3(X)= \ZZ \to H_3(X;\ZZ) = \ZZ$ 
is the multiplication by the degree $d>1$ of the covering $p: \widetilde X \to X$. 
It is well known that one can construct a $K(\pi_{1}(X),1)$ space $Z$ 
by adding a $4$-cell to kill the generator of $\pi_3(X) =\ZZ$, and adding further 
cells of dimension $\geq 5$ to kill the higher homotopy groups.
Then the inclusion $\phi: X \to  Z$ induces the identity on $\pi_1(X)$ and a surjection  $\phi_{*}: H_3(X;\ZZ) = \ZZ \to H_3(Z;\ZZ)= \ZZ/d\ZZ$. Therefore $X$ must be simply connected.
\end{proof}

In the proof of Proposition~\ref{prop:notrivialpi1} we
argue by contradiction using Corollary \ref{cor:simplyconnected} 
and the fact that $\pi_1(M)$ is not trivial. 

With the notation of Proposition~\ref{local structure},
we may assume that for arbitrarily large $D$ there exists $n\geq n_0(D)$ such that the
image of $\pi_1(B(x,\nu(x)))\to\pi_1(M_n)$ is  trivial for all $x\in {M_n}\setminus \textrm{int}(W_n)$ as well as for each component of
 $W_n$.

Then for all $x\in {M_n}\setminus \textrm{int}(W_n)$ we set:
$$\tr(x)=\sup \left\{ r\, \left\vert\begin{array}{l}
        \pi_1(B(x,r))\to\pi_1(M_n)\textrm{ is trivial and} \\
	B(x,r) \textrm{ is contained in } B(x',r') \textrm{ with }\\
	\textrm{ curvature }  \geq -\frac1{(r')^2}\textrm{ and } 
	\frac{\vol(B(x',r'))}{(r')^3}\leq 1/D
	 \end{array}
\right.
\right\}
$$

By hypothesis, we have $\tr(x)\geq \nu(x)$. The proof of 
Proposition~\ref{prop:notrivialpi1} follows by contradiction 
with the following assertion.

\begin{aff}\label{aff:covering}
There exists a covering of  $M_n$ by open sets $U_1,\ldots,$ $ U_p$  such that:
 \begin{itemize}
\item Each $U_i$ is either contained in some $B(x_i, \tr(x_i))$ or in a subset that
deformation retracts to a component of $\mathrm{int}(W_n)$.
In particular, $U_i$ is homotopically trivial in $M$.
\item The dimension of this covering is at most $2$.
\end{itemize}
\end{aff}

Since  $M$ is irreducible and non-simply connected, this contradicts Corollary \ref{cor:simplyconnected}.

To prove Assertion~\ref{aff:covering}, we define
$$
r(x)=\min \{\frac1{11}\tr(x),1\}.
$$

\begin{lem}\label{lem:cov:deux} Let $x,y\in { M_n}\setminus \operatorname{int}(W_n)$. 
If $B(x,{r(x)})\cap B(y,{r(y)})\neq\emptyset$, then
\begin{itemize}
\item[(a)]  $3/4\leq{r(x)}/{r(y)}\leq 4/3$;
\item[(b)]  $B(x,r(x))\subset B(y,4r(y))$.
\end{itemize}
\end{lem}

\begin{proof}
We may assume that $r(x)\leq r(y)$ and that $r(x)=\frac1{11}\tr(x)<1$.
From the triangle inequality, we get:
\[ \tr(x) \geq \tr(y)-r(x)-r(y) ,\]
hence
\[ 11\, r(x)=\tr(x)\geq 11r(y)-r(x)-r(y) \geq  9r(y).\]
Consequently, we have $1\geq {r(x)}/{r(y)}\geq 9/11\geq 3/4$, which shows (a).

Now (b) follows because $2\, r(x)+r(y)<4r(y)$.
\end{proof}

If $n$ is sufficiently large, we can choose points
$x_1,\ldots,x_q\in\partial W_n$ in such a way that
a tubular neighbourhood of each component of the
boundary of $W_n$ contains precisely one of the $x_j$'s and that
the balls $B(x_j,1)$ are disjoint, have volume
$\leq\frac1D$ and  sectional curvature close to
$-1$. Furthermore, we may  assume that
$B(x_j,1)$ is included in a submanifold $W'_n$ which contains
$W_n$ and can be retracted by deformation onto it. 
In particular $W'_n$ and $B(x_j,1)$ are homotopically trivial, since we have assumed that $W_n$ is. This implies that $\tr(x_j)$ is close to $1$.

 Moreover, 
for $n$ large enough, we may assume that
 $B(x_j, \frac23  r(x_j))$ contains an almost horospherical torus corresponding to a boundary component of $W_n$.
We can even  arrange for both components of
$B(x_j,  r(x_j))\setminus B(x_j, \frac 23r(x_j))$
to also contain a parallel almost horospherical torus, 
which allows to retract $W_n$ on the complement of  $B(x_j, \frac 23r(x_j))$.

We complete $x_1, x_2\ldots,x_q$ to a sequence $x_1,x_2,\ldots$ in ${ M_n}\setminus \mathrm{int} (W_n)$ such that the balls $ B(x_1,{\frac14 {r(x_1)}})$,
$B(x_2,{\frac14 {r(x_2)}}),\ldots$ are pairwise disjoint.

Such a sequence is necessarily finite, since ${M_n}\setminus \text{int}(W_n)$ is compact, and 
Lemma~\ref{lem:cov:deux} implies a positive local lower bound for the function $x\mapsto r(x)$. 
Let us choose a maximal finite sequence
 $x_1,\ldots,x_p$ with this property.

\begin{lem}\label{lem:seq}
The balls \(B(x_1,\frac23 r(x_1)),\ldots, B(x_p,\frac23 r(x_p))\)
cover  \linebreak ${M_n}\setminus \operatorname{int}(W_n)$.
\end{lem}
\begin{proof}
Let $x\in {M_n}\setminus \text{int}(W_n)$ be an arbitrary point. By maximality, there exists a
point $x_j$ such that $B(x,\frac14 {r(x)})\cap B(x_j,{\frac14
{r(x_j)}})\neq\emptyset$. From Lemma~\ref{lem:cov:deux}, we have
$r(x)\leq\frac43r(x_j)$ and $d(x,x_j)\leq\frac14(r(x)+r(x_j))\leq
\frac7{12}r(x_j)$, hence $x\in B(x_j,{\frac23 r(x_j)})$.
\end{proof}

Let us define $r_i\! :=r(x_i)$. If $W_{n,1},\ldots, W_{n,q}$
are the components of $W_n$, so that the almost horospherical torus
$\partial W_{n,i}\subset B(x_i,\frac 23 r_i)$, we set:
\begin{itemize}
 \item $ V_i:=B(x_i,r_i)\cup  W_{n,i}$, for $i=1,\ldots,q$. 
 \item $ V_i:=B(x_i,r_i) $, for $i > q$. 
\end{itemize}

 Furthermore,
each component of  $W_{n,i}$ can be retracted in order not to intersect $V_j$ when $j\neq i$.

The construction of the open sets $V_i$ and Lemma~\ref{lem:seq} 
imply the following:

\begin{lem} 
The open sets $V_1,\ldots,V_p$ 
cover $ M_n$.
%% Moreover for each pair 
%% $i, j \in \{1,\ldots,p\}$ 
%% the set $V_i \cup V_j$ is
%% homotopically trivial in $M$.
\end{lem}

Let $K$ be the nerve of the covering $\{V_i\}$.
We will use this covering and the complex $K$ to build the required map from $M$ to a $2$-dimensional complex. 
The idea is first to map
$M$ to  $K$  and then to improve this 
mapping by pushing it into the $2$-skeleton $K^{(2)}$ of $K$. 

The following Lemma shows that the dimension of $K$ is bounded above by a uniform constant.

\begin{lem}
There exists a universal upper bound $N$ on the number of open sets $V_i$ which intersect a given $V_k$.
\end{lem}
\begin{proof}
 If $V_i\cap V_k\neq\emptyset$, then $B(x_i,r_i)\cap B(x_k,r_k)\neq\emptyset$ and
 $B(x_i,r_i)\subset B(x_k,{2 r_i+r_k})\subseteq B(x_k,{4 r_k})$.
  On the other hand, for all $i_1\neq i_2$ such that $V_{i_1}$ and $V_{i_2}$ 
intersect $V_{k}$ one has 
 $d(x_{i_1},x_{i_2})\geq\frac14(r_{i_1}+r_{i_2})\geq\frac38r_k$. The number of $V_i$ intersecting $V_k$ is thus bounded above by:
 \[
     \frac{\operatorname{vol}(B(x_k,{4
r_k}))}{\operatorname{vol}(B(x_i,{\frac3{16} r_k}))}
     \leq
     \frac{\operatorname{vol}(B(x_i,{8 r_k}))}{\operatorname{vol}(B(x_i,{\frac3{16} r_k}))}
     \leq
     \frac{\operatorname{vol}(B(x_i,{11 r_i}))}{\operatorname{vol}(B(x_i,{\frac{r_i}{8} }))}.
\]
As $B(x_i,{11 r_i})$ is included in a ball $B(x',r')$ with curvature $\geq-\frac1{(r')^2}$, by Bishop-Gromov 
inequality this is bounded above by:
\[
     \frac{v_{-\frac1{(r')^2}}(11 r_i)}{v_{-\frac1{(r')^2}}(\frac{r_i}{8})}
     =
     \frac{v_{-1}(\frac{11 r_i}{r'})}{v_{-1}(\frac{r_i}{8 r'})}
	\leq N.
 \]
\end{proof}

Let  $\Delta^{p-1}\subset\mathbf R^{p}$ denote the standard
unit  simplex of dimension $p-1$.
With a partition of unity $(\phi_i)$ adapted to the $(V_i)$
and with certain metric properties, we construct a map:
\begin{equation*}%\label{tonerve}
f=\frac1{\sum_i\phi_i}(\phi_1,\ldots,\phi_p)\! :M_n \to
\Delta^{p-1}\subset \mathbf R^{p},
\end{equation*}

We view $K$ as a subcomplex of $\Delta^{p-1}$, so that
the range of  $f$ is contained in $K$, whose dimension is at most $N$. Moreover, $f$ maps the components of $W_n$ onto distinct vertices of the  $0$-skeleton  of $K$. We first
estimate the Lipschitz constant of the map $f\!:M_n\to K$, by choosing the $\phi_i$'s.

\begin{lem}\label{lem:boundlip}
There exists $L_N>0$ such that the partition of unity can be chosen so that the restriction $f\vert_{V_k}$ is
$\frac{L_N}{r_k}$-Lipschitz.
\end{lem}

\begin{proof}
 Let $\tau:[0,1]\to[0,1]$ be an auxiliary function with Lipschitz constant bounded by $4$, which vanishes in a neighbourhood of  $0$ and verifies
$\tau|_{[\frac13,1]}\equiv1$. Let us define
$\phi_k:=\tau(\frac1{r_k}d(\partial V_k,\cdot))$ on $V_k$ and let us extend it trivially on $M_n$.  Then $\phi_k$ is
$\frac4{r_k}$-Lipschitz.

Let $x\in V_k$. The functions $\phi_i$ have Lipschitz constant  $\leq \frac43\cdot\frac4{r_k}$, and  all $\phi_i$
vanish at  $x$ except at most $N+1$ of them. Since the functions 
\[ (y_0,\dots,y_N)\mapsto \frac{y_k}{\sum_{i=0}^N y_i} \]
are Lipschitz on  $\{y\in {\mathbf R}^{N+1} \mid  y_0\geq0, \ldots , y_N\quad
\text{and}\quad \sum_{i=0}^N y_i\geq1\}$, and each  $x\in M_n$
belongs to some $V_k$ with $d(x,\partial V_k)\geq \frac{r_k}3$,
the conclusion follows.
\end{proof}

We shall now inductively  deform $f$ by homotopy into the $3$-skeleton $K^{(3)}$, while keeping the local Lipschitz constant under control.

\begin{lem}\label{lem:homotopelowerskel}
For all $d\geq4$ and $L>0$ there exists $L'=L'(d,L)>0$ such that the following assertion holds true:

Let $g\! : \! M_n \to K^{(d)}$ be a
$\frac{L}{r_k}$-Lipschitz map defined on  $V_k$ and such that the pull-back of the open star of the vertex $v_{V_k}\in K^{(0)}$ is
contained in  $V_k$.  Then $g$ is homotopic rel  $K^{(d-1)}$ to a map $\tilde
g:M_n \to K^{(d-1)}$ with the same properties as $g$, $L$ being replaced by $L'$.
\end{lem}

\begin{proof}
It suffices to find a constant  $\theta=\theta(d,L) >0$
such that each $d$-simplex $\sigma\subset K$ contains a point
$z$ whose distance to  $\partial\sigma$ and to the image of $g$ is $\geq\theta$.
In order to push $g$ into the  $(d-1)$-skeleton, we compose it on $\sigma$
with the radial projection from $z$. This increases the Lipschitz constant by
a  multiplicative factor bounded above by a function of $\theta(d,L)$, 
and decrease the inverse image of the open stars of the vertices.

If $\theta$ does not satisfy the required property for some $d$-simplex $\sigma$, then $image(g)\cap int(\sigma)$ contains
a set of cardinality at least $C(d)\cdot\frac1{\theta^d}$ of  points whose
pairwise distances are $\geq\theta$. Let 
$A\subset M_n$ be a set containing exactly one point of the inverse image of each of  these points. 
Since $g$ maps $W_n$ into the  $0$-skeleton, $A\subset B(x_k,r_k)$.
As $g$ is
$\frac{L}{r_k}$-Lipschitz on $V_k$, the distance between any two distinct points in $A$ is bounded
below by $\frac1{L}r_k\theta$ . Hence the cardinal of  $A$
is bounded above by
\begin{multline*}
\frac{\vol(B(x_k,r_k)}{\vol(B(y,\frac{r_k\theta}{2L}))}
\leq
\frac{\vol(B(y,2\,r_k)}{\vol(B(y,\frac{r_k\theta}{2L}))}
\leq
\frac{v_{-\frac1{(r')^2}}(2\, r_k)}{v_{-\frac1{(r')^2}}(\frac{r_k\theta}{2L})}
=
\\
\frac{v_{-1}(2\frac{r_k}{r'})}{v_{-1}(\frac{\theta}{2L}\frac{r_k}{r'})}\leq C\left(\frac{L}{\theta}\right)^3,
\end{multline*}
where $y$ is any point in $A$. In order to apply Bishop-Gromov, 
we used the fact that $B(x_k,11 r_k)$ is included in a ball of radius  $r'$ with curvature
$\geq -1/(r')^2$.
The inequality $C(d)\cdot\frac1{\theta^d}\leq C\cdot(\frac{L}{\theta})^3$
gives a positive lower bound $\theta_0(d,L)$ for $\theta$.
Consequently, any $\theta<\theta_0$ has the desired property.
\end{proof}

\begin{lem}\label{lem:volsubcubic}
There exists a constant $C$ such that if $D$ is large enough, then
\[ \vol( B(x_i,r_i) )\leq C\frac1D r_i^3 \qquad\textrm{ for all } i.\]
\end{lem}

\begin{proof}
We know that $\vol (B(x_i,\nu_{x_i }))\leq \frac 1D \nu_{x_i}^3$.
Furthermore, $r_i\geq \frac{\nu_{x_i}}{11}$ and
$B(x_i,r_i)$ is included in a ball $B(x',r')$ with curvature 
$\geq -\frac1{r'^2}$. As $r'\geq r_i$, the curvature on  $B(x_i,r_i)$
is $\geq -\frac1{r_i^2}$. The Bishop-Gromov inequality gives:
$$
\frac{\vol(B(x_i,\frac{\nu_{x_i}}{11}))}{v_{-\frac1{r_i^2}} ( \frac{\nu_{x_i}}{11}) }
\geq \frac{\vol(B(x_i,r_i))}{v_{-\frac1{r_i^2}} ( r_i )}.
$$
Equivalently,
$$
{\vol(B(x_i,\frac{\nu_{x_i}}{11}))}   
\geq \frac{\vol(B(x_i,r_i))}{v_{-1} ( 1 )}  {v_{-1 } ( \frac{\nu_{x_i}}{11\, r_i}) }
\geq
\vol( B(x_i,r_i) ) \frac1C \left(\frac{\nu_{x_i}}{r_i}\right)^3,
$$
for some uniform $C>0$.
Hence 
$$
\vol( B(x_i,r_i) )  \leq C \left(\frac{r_i}{\nu_{x_i}}\right)^3
\vol( B(x_i,\frac{\nu_{x_i}}{11}) )\leq Cr_i^3 \frac 1D.
$$
% % Hence
% % \[
% % \frac1D\geq \frac{\vol(B(x_i,\frac{\nu_{x_i}}{11}))}{\nu_{x_i}^3 }
% % \geq
% %   \frac{\vol(B(x_i,r_i))}{r_i^3}
% % \frac
% % {v_{-1 } ( \frac{\nu_{x_i}}{11\, r_i}) }
% % {v_{-1} ( 1 ) \left(\frac{\nu_{x_i}}{r_i}\right)^3}
% % \geq
% % \frac{\vol(B(x_i,r_i))}{r_i^3}\frac1C
% % \]
% % where the existence of  $C$ comes from the fact that  $\frac{\nu_{x_i}}{11\, r_i}\leq 1$
% % and $v_{-1}(r)$ and $r^{3}$ are comparable for $r\in [0,1]$.
\end{proof}

Finally we push $f$ into the  $2$-skeleton.

\begin{lem}\label{lem:2skeleton}
For a suitable choice of $D>1$, there exists a map $f^{(2)}\!:  M_n \to K^{(2)}$ such that:
\begin{enumerate}
\item $f^{(2)}$ is homotopic to $f$ rel $K^{(2)}$.
\item The inverse image of the open star of each vertex $v_{V_k}\in K^{(0)}$ is contained in  $V_k$.
\end{enumerate}
\end{lem}

\begin{proof}
The inverse image by  $f$ of the open star of the vertex
$v_{V_k}\in K^{(0)}$ is contained in  $V_k$.
Using Lemma~\ref{lem:homotopelowerskel} several times, we find a map
$f^{(3)}: M_n \to K^{(3)}$ homotopic to $f$ and a universal constant $\hat L$
such that $(f^{(3)})^{-1}(star(v_{V_k}))\subset V_k$ and $f^{(3)}_{\vert V_k}$
is $\frac{\hat L}{r_k}$-Lipschitz.

It now suffices to show that no $3$-simplex $\sigma\subset K$
can lie entirely in the image of $f^{(3)}$. Indeed, once we know this, we can push $f^{(3)}$ into
the $2$-skeleton of $K$ using a central projection in each simplex, with centre in the complement of this image. Note that here no metric estimate is required in the conclusion.

Let us thus assume that there exists a $3$-simplex $\sigma$ contained in the image of $f^{(3)}$.
The inverse image of $int(\sigma)$ by $f^{(3)}$ is a subset of the
intersection of the $V_j$'s such that $v_{V_j}$ is a vertex of $\sigma$.
Let $V_k$ be one of them. 
As $\vol(f^{(3)}(V_k))\leq \vol(f^{(3)}(B(x_k,r_k)))$,
Lemma~\ref{lem:volsubcubic} yields:
\[
\vol(image(f^{(3)})\cap\sigma)\leq \vol(f^{(3)}(V_k))
 \leq ( \frac{\hat L}{r_k})^3 \vol( B(x_k,r_k) )
  \leq C \hat L^3\frac1D
\]
with uniform constants $C$ and $\hat L$. Hence, if $D$ is
sufficiently large, then one has $\vol(image(f^{(3)})\cap\sigma)<\vol(\sigma)$.
\end{proof}

The retraction  used to push $f$ into the $2$-skeleton does not involve the $0$-skeleton of $K$. 
As a consequence, the  inverse images of the open stars of the vertices $v_k$ still satisfy  
$(f^{(2)})^{-1}(star(v_{V_k}))\subset V_k$, and thus $(f^{(2)})^{-1}(star(v_{V_k})) 
$ is
homotopically trivial in $M$. This proves the assertion and ends the proof by contradiction of Proposition~\ref{prop:notrivialpi1}.
\qed

\subsection{End of the proof of Theorem \ref{TH:GRAPH}}
Proposition~\ref{prop:notrivialpi1} implies:

\begin{corol}\label{cor:nontrivial}
There exists $D_0>0$ such that if $D>D_0$ and $n\geq n_0(D)$,
 then there exists a compact submanifold ${\mathcal V}_0 \subset M_n$ with the following properties:
\begin{itemize}
\item[(i)]  ${\mathcal V}_0$ is either a connected component 
of $W_n$ or a tubular neighbourhood of the soul of the local model of some point
 $x_0\in { M_n}\setminus \operatorname{int} (W_n)$. 
\item[(ii)] ${\mathcal V}_0$ is a solid torus, a thickened torus or the twisted  $I$-bundle on the Klein bottle. 
\item[(iii)]  ${\mathcal V}_0$ is homotopically non-trivial in $M_n$.
\end{itemize}
\end{corol}

\begin{proof}
We recall that each component of $W_n$ is a solid torus.
If one of them is homotopically non-trivial, then we choose it.
Otherwise, by
Proposition~\ref{prop:notrivialpi1}, there exists a  point 
$x_0\in {M_n}\setminus \operatorname{int} (W_n)$  such that 
$B(x_0,\nu_{x_0})$ is homotopically non-trivial; 
one of the remarks following Proposition~\ref{local structure} shows that
$B(x_0,\nu_{x_0})$ is necessarily a solid torus, 
a thickened torus or a twisted $I$-bundle over the  Klein bottle. Indeed,
$S_0$ can be neither a point nor a $2$-sphere, otherwise $B(x_0,\nu_{x_0})$
would be homeomorphic to  $B^3$ or $S^2\times I$, which have trivial fundamental group.
\end{proof}

As ${\mathcal V}_0$ is not contained in any $3$-ball,  
each component $Y$ of its complement is
irreducible, hence a \emph{Haken manifold} 
whose boundary is a union of tori. In particular, $Y$ admits a geometric
decomposition. Here is an important consequence of Thurston's hyperbolic 
Dehn filling theorem ({cf.~\cite[Prop.~10.17]{blp},
\cite[Prop.~9.36]{BMP}}):

 \begin{prop}\label{prop:haken}
Let $Y$ be a Haken 3-manifold whose boundary is a union of tori.
Assume that any manifold obtained from   $Y$ by Dehn filling has vanishing simplicial volume. Then $Y$ is a graph manifold.
\end{prop}

In order to prove that $M_n$ is a graph manifold, it is sufficient to show that each component of $M_n\setminus {\mathcal V}_0$ is a graph manifold.
To conclude the proof of Theorem~\ref{TH:GRAPH}, it suffices to show the following proposition:

\begin{prop}\label{aff:volumesimp}
For $n$ large enough, one can find a submanifold ${\mathcal V}_0$ as above such that 
every Dehn filling on each component  $Y$ of $M_n \setminus \operatorname{int}({\mathcal V}_0)$ has vanishing simplicial volume.
\end{prop}

%\medskip

We choose the set ${\mathcal V}_0$ as follows: 
\begin{itemize}
\item If some component of $W_n$ is homotopically non-trivial, then we choose it as ${\mathcal V}_0$.
\item  If all components of $W_n$ are homotopically trivial, then there exists a point $x\in 
{M_n} \setminus \operatorname{int} (W_n)$ such that 
$B(x,\nu_{x})$ is homotopically non-trivial. We choose $x_0 \in {M_n}\setminus \operatorname{int} (W_n)$ such that
$$
\nu_0=\nu_{x_0}\geq \frac12\sup\{ \nu(x) \mid  \pi_1(B(x,\nu(x)))\to\pi_1(M_n)\textrm{\ is non-trivial}\}.
$$

Let $S_0$ be the soul of the local model $B(x_0,\nu_{0})$. We choose
${\mathcal V}_0$ to be the  metric open $\delta$-neighbourhood 
with $0 <\delta <\frac{\nu_0}{D}$.
After possibly shrinking $W_n$, one has ${\mathcal V}_0\cap W_n=\emptyset$,
as $\nu_0\leq 1$.
\end{itemize}

\medskip

We say that a subset $U \subset M_n$ is
\bydef{virtually abelian relatively to} ${\mathcal V}_0$ if the image in 
$\pi_1(M_n\setminus {\mathcal V}_0)$ of the fundamental group of each connected component of $U \cap (M_n\setminus {\mathcal V}_0)$ 
is virtually abelian.

We set:
$$
\ab(x)=\sup \left\{ r\, \left\vert\begin{array}{l}
                                 B(x,r)\textrm{ is virtually abelian relatively to } {\mathcal V}_0 \textrm{ and }\\
				 B(x,r) \textrm{ is contained in a ball } B(x',r') \textrm{ with }\\
				\textrm{curvature }  \geq -\frac1{(r')^2}\textrm{ and } 
				\frac{\vol(B(x',r'))}{(r')^3}\leq 1/D
				 \end{array}
\right.
\right\}
$$
and
$$
r(x)=\min \{\frac1{11}\ab(x),1\}.
$$
We are now led to prove the following assertion:

\begin{aff}\label{aff:abelcov}
With this choice of ${\mathcal V}_0$, 
for $n$ large enough, $M_n$ can be covered by
a finite collection of open sets $U_i$ such that:
\begin{itemize}
 \item Each $U_i$ is either contained in a component of $W_n$ or in a ball $ B(x_i,r(x_i))$ for some $x_i\in M^-_n(\varepsilon_n)$. 
In particular, $U_i$ 
 is virtually abelian relatively to $ {\mathcal V}_0$.
 \item The dimension of this covering is not greater than $2$, and it is zero on ${\mathcal V}_0$.
\end{itemize}
\end{aff}

Let us first show why this assertion implies Proposition~\ref {aff:volumesimp}.

\begin{proof}
The covering described in the assertion induces naturally a covering on every closed and orientable manifold
 $\hat Y$, obtained by gluing solid tori to
$\partial Y$.  It is a $2$-dimensional covering by open sets which are virtually abelian and thus amenable in $\hat Y$. Gromov's vanishing theorem \cite[\S 3.1]{GromovIHES}, see also
\cite{Ivanov}, then implies that the simplicial volume of $\hat Y$ vanishes, which proves 
 Proposition~\ref {aff:volumesimp}.
\end{proof}

We now prove Assertion~\ref{aff:abelcov}.
The argument for the construction of a $2$-dimensional 
covering by abelian open sets is similar to the one used in the proof of 
Proposition~\ref{prop:notrivialpi1}, replacing everywhere
the triviality radius $\tr$ by the abelianity radius $\ab$. 
There are, however, a few differences, which we now point out.

If ${\mathcal V}_0$ is a connected component of $W_n$, then for $n$
large enough we choose points $x_0,x_1,\ldots,x_q\in\partial W_n$, with $x_0 \in \partial {\mathcal
V}_0$ in such a way that
\begin{itemize}
 \item  Every boundary component  of $W_n$ contains
exactly one of the  $x_j$'s.
\item The balls $B(x_j,1)$ are pairwise disjoint.
\item Every $B(x_j,1)$ has normalised volume $\leq\frac1D$ and sectional curvature close to $-1$.
\item Every $B(x_j,1)$ is contained in a thickened torus (which implies that this ball is abelian).
\end{itemize}
Furthermore, going sufficiently far in the cusp and taking
$n$ large enough, one can assume that $B(x_j, \frac19 r(x_j))$ contains an almost horospherical torus corresponding to a boundary component of $W_n$.  
In this case the proof previously done applies without any change, since the dimension of the original covering and all those obtained by shrinking 
is zero on $W_n$ (or on a set obtained by shrinking $W_n$).

From now on we shall assume that \emph{all connected components of $W_n$ are homotopically trivial}. 
We then choose $x_0 \in {\mathcal V}_0 \subset {M_n}\setminus \operatorname{int} (W_n)$ as above, and points
$x_1,\ldots,x_q\in\partial W_n$ as before. 

We complete the sequence $x_0, x_1,\ldots,x_q$ to a maximal finite sequence
$x_0,x_1,x_2,\ldots, x_p$ such that the balls $B(x_i,{\frac14 {r(x_i)}})$
 are disjoint.

We set $r_i=r(x_i)$,  and, if $W_{n,1},\ldots W_{n,q}$ are the connected components of
$W_n$, then we set
\begin{itemize}
 	\item $V_0:=B(x_0,r_0) $.
	\item $V_i:=B(x_i,r_i)\cup W_{n,i}$, for $i=1,\ldots,q$.
	\item $V_i:=B(x_i,r_i) \setminus {\mathcal V}_0$ for $i=q+1,\ldots, p$.
\end{itemize}
{
After possibly shrinking $W_{n}$, we have ${\mathcal V}_0\cap B(x_i,r_i)=\emptyset$ for $i=1,\ldots,q$,
since $r_i\leq 1$ and $\mathcal V_0\cap W_n=\emptyset$. 
It follows that $\mathcal V_0\cap V_i=\emptyset$ for $i\neq 0$.
}

Under the hypothesis that the $W_n$ are homotopically trivial, 
the following two lemmas deal with the difference with the previous proofs.

\begin{lem} \label{lem:rkover3}
Each $x\in M_n$ belongs to some $V_k$ such that $d(x,\partial
V_k)\geq \frac 13 r_k$.
\end{lem}

This lemma is used in the control of the Lipschitz constant of  the characteristic map, Lemma~\ref{lem:boundlip}.
In order to prove it, we begin with the following remark:
\begin{rem}
 If $n$ is large enough, we have ${\mathcal V}_0\subset B(x_0,\frac{r_0}9)$.
\end{rem}

The bounds on the
diameter of $S_0$, the distance to the base point, and the radius of the neighbourhood,
give ${\mathcal V}_0\subset B(x_0,\frac{3\nu_{x_0}}D)$. Then the remark follows.

\begin{proof}[Proof of Lemma~\ref{lem:rkover3}]
If  $x\in W_n$, then there is nothing to prove. We thus assume that
$x\in M_n\setminus W_n$.
If $x\in B(x_0,{\frac23r_0})$ we may choose $k=0$.
Let us then assume that $x\not\in B(x_0,{\frac23r_0})$.
There exists $k$ such that $x\in B(x_k,{\frac23r_k})$.
If $V_k$ and $V_0$ are disjoint, then we are done. Hence
we assume that $V_k\cap V_0\neq\emptyset$.
By the previous remark, one has:
\[ d(x,{\mathcal V}_0)\geq d(x,x_0)-\frac19r_0
\geq \frac23r_0-\frac19r_0 \geq \frac34\cdot\frac59r_k>\frac13r_k.
\]
This implies that  $d(x,\partial V_k)\geq \frac13r_k$.
\end{proof}

The second difference is that  the triviality radius 
satisfies $\tr(x_i)\geq \nu_{x_i}$ by construction. Here we shall prove that
$\ab(x_i)\geq c\,\nu_{x_i}$ for a uniform $c>0$. This is used in the proof of  Lemma~\ref{lem:volsubcubic}, where the inequality
$\vol(B_i)\leq C\frac1D r_i^3$ will still be true, but with a different constant $C$.

\begin{lem}
\label{lem:uniformabcnu}
There exists $c>0$ such that $r_i\geq c\, \nu_{x_i}$ for all $i$.
\end{lem}

\begin{proof}
One has $r_0\geq \frac1{11}\nu_{x_0}$by construction. For all  $i>0$, if
$B(x_i,\frac{\nu_{x_i}}{11})\cap {\mathcal V}_0=\emptyset$, then
$r_i\geq \frac1{11}\nu_{x_i}$. Hence we assume
$B(x_i,\frac{\nu_{x_i}}{11})\cap {\mathcal V}_0\neq\emptyset$, and we claim that
 $d(x_i,{\mathcal V}_0)>c'\, \nu_{x_i}$ for a uniform  $c'>0$.
Since ${\mathcal V}_0\subset B(x_0,\frac19 r_0)$:
 \begin{equation}
\label{eqn:distancenu0} 
  d(x_i,{\mathcal V}_0)\geq d(x_i,x_0)-\frac19 r_0
 \geq \frac 14r_0-\frac19r_0 >\frac 18r_0 \geq \frac 1{100}\nu_{x_0}.
 \end{equation}
We distinguish two cases, according to whether ${\mathcal V}_0$ is 
contained in $B(x_i,\nu_{x_i})$ or not.

If ${\mathcal V}_0 \subset B(x_i,\nu_{x_i})$, the image of
$\pi_1(B(x_i,\nu_{x_i}))\to\pi_1(M_n)$ cannot be trivial,
since the image of $\pi_1({\mathcal V}_0)\to \pi_1(M_n)$ is not. 
In addition,
\begin{equation}
 \label{eqn:nu0nui}
 \nu_{x_i}\leq 2\,\nu_{x_0},
 \end{equation}
by the choice of $\nu_{x_0}$. Equations (\ref{eqn:distancenu0}) and
(\ref{eqn:nu0nui}) give $d(x_i,{\mathcal V}_0) > \nu_{x_i}/200$.

If ${\mathcal V}_0 \not\subset B(x_i,\nu_{x_i})$, then since
${\mathcal V}_0 \cap B(x_i,\frac{\nu_{x_i}}{11})\neq\emptyset$, we have
$$
\diam ({\mathcal V}_0)+\frac{\nu_{x_i}}{11}\geq \nu_{x_i}.
$$
Consequently, 
$\diam({\mathcal V}_0)/\nu_{x_i}\geq 10/11$. As $\diam({\mathcal V}_0)\leq 3\nu_{x_0}/D$, we deduce that
$$
\frac{\nu_{x_0}}{\nu_{x_i}}\geq \frac{10}{33} D.
$$
Combining with~(\ref{eqn:distancenu0}), we get
$d(x_i,{\mathcal V}_0)\geq \frac{D}{330} \nu_{x_i}$.

For $D> 30$, this contradicts ${\mathcal V}_0 \cap
B(x_i,\frac{\nu_{x_i}}{11})\neq\emptyset$.
\end{proof}

\newpage
\section{Ricci flow with surgery}\label{chir}
In order to apply Theorem~\ref{TH:GRAPH}, we shall
need the following straightforward consequence of Perelman's work~\cite{PerelmanI,
PerelmanII,PerelmanIII}:
\begin{theo}\label{existence sequence}
Let $M$ be a closed, orientable, \irr\  $3$-\var.
\begin{enumerate}
\item[(1)] If $\pi_1M$ is finite, then $M$ is spherical.
\item[(2)] If $\pi_1M$ is infinite, then for every riemannian metric $g_0$ on $M$,
there exists an infinite sequence of riemannian metrics $g_1,\ldots,g_n,\ldots$ with
the following properties:
\begin{enumerate}
\item[(i)] The sequence $(\hat R(g_n))_{n\ge 0}$ is nondecreasing. In particular, it has a limit,
which is greater than or equal to $\hat R(g_0)$.
\item[(ii)] The sequence $(\vol(g_n))_{n\ge 0}$ is bounded.
\item[(iii)] Let $\varepsilon>0$ be a real number and $x_n\in M$ be a sequence such that
for all $n$, $x_n$ is $\varepsilon$-thick with respect to $g_n$. Then the sequence $(M,g_n,x_n)$ subconverges in the $\mathcal C^2$ topology towards some hyperbolic pointed manifold.
\item[(iv)] The sequence $g_n$ has controlled curvature in the sense of Perelman.
\end{enumerate}
\end{enumerate}
\end{theo}

In this section we explain how to deduce Theorem~\ref{existence sequence} from Perelman's results on the Ricci flow. We refer to~\cite{KleinerLott} and~\cite{MorganTian} for the details.

Let $M$ be a closed, orientable $3$-manifold.
R.~Hamilton introduced in~\cite{Hamilton3} the following equation:
$${\frac{dg}{ dt}}=-2\Ric (g),$$
where the unknown $g=g(t)$ is a family of riemannian metrics on $M$ depending on a 
time parameter $t\in \textbf{R}$. A  \bydef{Ricci flow} is a solution to this  equation.

In~\cite{PerelmanII}, Perelman constructs an object he calls \emph{Ricci flow with $\delta$-cutoff},
also known as \emph{Ricci flow with surgery}. It can be viewed as a $1$-parameter family
of (possibly disconnected) riemannian manifolds $(M(t),g(t))$ which satisfies Hamilton's equation
in a weak sense.
The topology of the manifold $M(t)$ is allowed to change at a discrete set of times, the change
being a connected sum decomposition into prime factors, as well as $RP^3$ or $S^2\times S^1$ factors,
and removing components that are spherical or diffeomorphic to $S^2\times S^1$.

Perelman~\cite[\S 1--5]{PerelmanII} (see also~\cite[\S 58--80]{KleinerLott},
\cite[Chapters 14--17]{MorganTian})  shows that for every riemannian metric $g_0$ on $M$, there exists a Ricci flow with
surgery satisfying the initial condition $(M(0),g(0))=(M,g_0)$. It may happen that $M(t)$ becomes
empty for some finite time $t$; in this case, $M$ is a connected sum of spherical manifolds
and copies of $S^2\times S^1$. Such a Ricci flow with surgery is said to become \bydef{extinct}.

From now on we assume that $M$ is \irr. Thus, if some Ricci flow with surgery with initial
manifold $M(0)=M$ becomes extinct, then $M$ is spherical. If $M$ is not diffeomorphic to
$S^3$ and $(M(t),g(t))$ is a Ricci flow with surgery such that $M(0)=M$ and which does not
become extinct, then for each time $t$, the manifold $M(t)$ has exactly one component
diffeomorphic to $M$, the others being copies of $S^3$. Thus we get a $1$-parameter family
of metrics on $M$, which we still denote by $g(t)$, defined for all $t\ge 0$. (There is some
freedom for the choice of diffeomorphisms between $M$ and the various $M(t)$'s, but the
following discussion does not depend on this choice.)

If the metric $g_0$ has positive scalar curvature, then a maximal principle argument
shows that any Ricci flow with surgery with initial metric $g_0$ becomes extinct in finite time. Thus $M$ is spherical. The same conclusion holds if $\pi_1M$ is finite
(see~\cite{PerelmanIII}, \cite[Chapter~18]{MorganTian}, \cite{cm:question}.)

If $\pi_1M$ is infinite, then Ricci flow with surgery cannot become extinct,
and $M$ cannot admit any metric of positive scalar curvature. Let $g_0$
be any riemannian metric on $M$. By scaling, we get a \bydef{normalised} metric $\hat g_0$
(i.e.~the absolute value of its sectional curvature is  bounded above by $1$, 
and each ball of radius  $1$ has volume greater than or equal to half of the volume of the Euclidean unit ball.) Starting with the initial metric $\hat g_0$, we get a family of metrics
$\{g(t)\}_{t\ge0}$. For all integers $n\ge 1$, set
$g_n:=(4n)^{-1}g(n)$. Then it is proved in~\cite[\S 6--7]{PerelmanII} (see~\cite[\S 81--89]{KleinerLott}
for details) that the sequence $\{g_n\}_{n\ge 0}$ satisfies Properties~(ii)--(iv) of the
conclusion of Theorem~\ref{existence sequence}.\footnote{
An immaterial difference between our statement and Perelman's is that we use the rescaling
factor $(4t)^{-1}$ instead of $t^{-1}$ in order to get limits of sectional curvature $-1$ rather than $-1/4$.}
Moreover, the function $t\mapsto \hat R(g(t))$ is nondecreasing~\cite[\S 7.1]{PerelmanII}. Since $\hat R$ is scale invariant,
we have $\hat R(g_0)= \hat R(\hat g_0)$, and $\hat R(g_n)=\hat R(g(n))$ for all $n\ge 1$.
Hence the sequence $(\hat R(g_n))_{n\ge 0}$ is nondecreasing. This completes
the proof of Theorem~\ref{existence sequence}.

\section{Applications}\label{sec:geom}

\subsection{A sufficient condition for hyperbolicity}\label{sec:hyp}

The following proposition is Assertion (1) of Theorem~\ref{th:minoration}:

\begin{prop}\label{prop:hyp} Let $M$ be a closed, orientable and irreducible $3$-manifold. Suppose
that the inequality 
 $\overline{R}(M) \leq -6 V_0(M) ^{2/3}$ holds. Then $M$ is hyperbolic and the hyperbolic metric realises $\overline{R}(M)$.
\end{prop}

\begin{proof} 
Let $H_0$ be a hyperbolic manifold homeomorphic to the complement of a link $L_0$ in $M$ and whose volume realises  $V_0(M)$. To prove Proposition~\ref{prop:hyp}, it is sufficient to show that $L_0$ is empty. Let us assume that it is not true and prove that  $M$ carries a metric $g_\varepsilon$ such that
$\vol(g_\varepsilon) < V_0(M)$ and $R_{\textrm{min}}(g_\varepsilon)\ge -6$. This can be done by a direct construction as in~\cite{Anderson1}. We give here a different argument relying on Thurston's hyperbolic Dehn filling theorem.

If $L_0\neq \emptyset$, then we consider the orbifold $\mathcal{O}$ with underlying space $M$,  singular locus $L_0$ local group $\ZZ/n\ZZ$ with $n
>1$ sufficiently large so that the orbifold carries a hyperbolic structure,
by the hyperbolic Dehn filling theorem \cite{Thurston} (cf.~\cite[Appendix B]{bp}). We then desingularise the conical metric on $M$ corresponding to the orbifold structure, in a tubular neighbourhood of $L_0$:

\begin{lem}[Salgueiro \cite{Salgueiro}] \label{lem:desingularisation}
For each $\varepsilon >0$ there exists a Riemannian metric
$g_{\varepsilon}$ on $M$ with sectional curvature bounded below by
 $-1$ and such that $\vol(M,g_{\varepsilon}) < (1 +\varepsilon)^{\frac{5}{2}}\vol(\mathcal{O})$. 
\end{lem}

For completeness we give the proof of this lemma, following \cite[Chap. 3]{Salgueiro}.

\begin{proof}
Let $g_{}$ be the  hyperbolic cone metric on $M$ induced by the hyperbolic orbifold $\mathcal{O}$. 
Let $\mathcal{N} \subset \mathcal{O}$ be a tubular neighbourhood of radius $r_0>0$ around the singular locus $L_0$. In $\mathcal{N}$  the local expression of the singular 
metric $g_{}$ in Fermi (cylindrical) coordinates is: 
$$d s^2=  d r^2+ \left(\frac{1}{n}{\sinh(r)}\right)^2 d\theta^2 + \cosh^2(r) dh^2,
$$ 
where $r\in (0,r_0)$  is the
distance to $L_0$, $h$ is the length parameter along $L_0$,
and $\theta\in(0,2\pi)$ is the rescaled angle parameter.

%We  deform the metric $g_{}$ in  $\mathcal{N}$ 
% of radius $r_0$, for some $r_0>0$ sufficiently small, around the singular locus $L_0$.
The deformation depends only on the parameter
$r$ and  consists in replacing the  metric $g_{}$ by a smooth
metric $g'$ which coincides with $g_{}$ outside of $\mathcal{N}$, and has in $\mathcal{N}$ the form 
$$
 ds^2=dr^2+ \phi^2(r)d\theta^2 + \psi(r)^2dh^2,
$$
where for some $\delta=\delta(\varepsilon)>0$ sufficiently small the functions
$$\phi, \psi:[0,r_0-\delta]\to[0,+\infty)$$ are smooth and
satisfy the following properties:
\begin{itemize}
\item [(1)] In a neighbourhood of $0$, $\phi(r)=r$ and $\psi(r)$ is constant. 
\item [(2)] In a neighbourhood of $r_0-\delta$, $\phi(r)=\frac{1}{n}{\sinh(r+\delta)}$ and 
$\psi(r)=\cosh(r+\delta)$. 
\item [(3)] $\forall r \in (0,r_0-\delta), \frac{\phi''(r)}{\phi(r)} \leq 1+\varepsilon, \, \, \, \frac{\psi''(r)}{\psi(r)} \leq 1+\varepsilon$ and $\frac{\phi'(r)\psi'(r)}{\phi(r)\psi(r)} \leq 1+\varepsilon.$
\end{itemize}
The new metric is non-singular by (1), it matches the previous one away from $\mathcal N$ by (2)
and has sectional curvature $\geq -1-\varepsilon$ by (3).

First we deal with the construction of $\phi$. Let $r_1=r_1(\delta)>0$ be defined by $r_1=\frac1n\sinh(r_1+\delta)$. Notice that $r_1\thickapprox \frac 1{n-1}\delta$.
The function $\phi$ is a smooth modification in a neighbourhood of $r_1$ of the piecewise smooth function
$$
r\mapsto \left\{  \begin{array}{ll}
 r &\textrm{on } [0,r_1] \\
\frac1n\sinh(r+\delta) &\textrm{on }  [r_1,r_0-\delta]
\end{array}
\right.
$$
so that:
\begin{itemize}
 \item On $[0,r_1]$, $\phi\geq \frac{r}{1+\varepsilon}$, $\phi'\leq 1$, $\phi''\leq 0$.
 \item On $[r_1,r_0-\delta ]$, $\phi\geq \frac 1n\sinh(r+\delta) \frac{1}{1+\varepsilon}$, $\phi'\leq 
\frac 1n\cosh(r+\delta)$, $\phi''\leq \frac 1n\sinh(r+\delta) $.
\end{itemize}
We choose  $\psi$ satisfying: 
\begin{itemize}
\item On $[0,r_1]$, $\psi$ is constant.
 \item On $[r_1,r_0-\delta ]$, $\psi\geq \cosh(r+\delta)$, $\psi'\leq  \sinh(r+\delta)$, 
$\psi''\leq \cosh(r+\delta) (1+\varepsilon)$.
\end{itemize}
Notice that $\psi'(r_1)=0$ and $\psi'(r_0-\delta)=\sinh(r_0)$, so for a given $\varepsilon> 0$, one has to choose $\delta$ sufficiently small
to achieve the required bound on $\psi''$.

As $\delta \to 0$,  $\vol(M,g') \to \vol(\mathcal{O})$, since $\phi \to \frac{1}{n}{\sinh(r)}$ and 
$\psi \to \cosh^2(r)$. So given $\varepsilon >0$, for a choice of $\delta$ sufficiently small, one obtains a smooth Riemannian metric $g'$ on $M$ with sectional curvature $\geq -1- \varepsilon$
and volume $\vol(M,g') \leq (1+\varepsilon)\vol(\mathcal{O})$. Then the rescaled metric 
$g_{\varepsilon} = \sqrt{1+ \varepsilon}\,g'$ on $M$ has  sectional curvature $\geq -1$ and volume 
$\vol(M,g_{\varepsilon}) \leq  (1+\varepsilon)^{\frac{5}{2}}\vol(\mathcal{O})$.
 \end{proof}

 As $\vol(\mathcal{O}) < \vol(H_0)$, 
for  $\varepsilon>0$ sufficiently small we obtain a Riemannian metric on
 $M$ such that $\vol(M,g_{\varepsilon}) < \vol(H_0)$ and
$R_{\textrm{min}}(g_{\varepsilon}) \geq -6$. In particular 
$$
\hat{R}(g_{\varepsilon}) > -6
\vol(H_0)^{2/3} = -6 V_0(M) ^{2/3}
$$ 
which contradicts the hypothesis. The link $L_0$ is thus empty and we have  $M=H_0$.

\end{proof}

\subsection{Proof of Theorem~\ref{th:minoration}}\label{sec:minoration}

Thanks to previous section, there just remains to show Assertion (2). If
$\pi_1M$ is finite, then by Theorem~\ref{existence sequence}(1), $M$ is spherical,
hence a graph manifold. From now on we assume that $\pi_1M$ is infinite.
In particular, $M$ is not simply connected.

By assumption, there exists a riemannian metric
$g_0$ on  $M$ such that $ -6V_0(M) ^{2/3} < \hat{R}(g_0)$.  Applying Theorem~\ref{existence sequence}(2), we get a sequence of metrics $\{g_n\}$ satisfying properties~(i)--(iv) of this Theorem.
Note that properties~(ii) and~(iv) are respectively Hypotheses~(1) and (3) of Theorem~\ref{TH:GRAPH}. Next we check Hypothesis~(2), which is the content of the
following lemma:

\begin{lem}\label{lem:volume}
If $H$ is any hyperbolic $3$-manifold which appears as a pointed $\mathcal C^2$-limit
of the some subsequence of $(M,g_n)$, then $\vol(H) < V_0(M)$.
\end{lem}

\begin{proof}
 Looking for a contradiction, we assume $\vol(H) \geq V_0(M)$.

By monotonicity of $\hat R(g_n)$ and choice of $g_0$, we have
$$\lim \hat{R}(g_n) \ge \hat{R}(g_0) > -6 V_0(M) ^{2/3}.$$

Let $\xi>0$. We choose a compact core  $\bar H(\xi)$ of $H$ such that $\vol (\bar H(\xi))\geq (1-\xi) \vol(H)$. 
Since the convergence is $\mathcal C^2$, there exists, for $n$ sufficiently large, a submanifold  $\bar H_n(\xi)\subset M_n$ with volume at least $ (1-\xi)^2 \vol(H)$ and whose scalar curvature is  less
than or equal to $-6(1-\xi) $. Thus, for $n$  large,  $R_{\textrm{min}}(g_n) \leq -6(1-\xi) $ and $\vol(M_n) \geq \vol(\bar H_n(\xi)) \geq (1-\xi)^2 \vol(H)$. Letting $\xi\to 0$ we get
$$\lim \hat{R}(g_n) \leq -6 \vol(H)^{2/3} \leq -6V_0(M)^{2/3},
$$ which gives the desired contradiction. 
\end{proof}

Hence we can apply Theorem~\ref{TH:GRAPH}. It follows that $M$  contains an incompressible torus or is a graph manifold. In the latter case, $M$
contains an incompressible torus or is Seifert fibred.\qed

\subsection{Proof of Theorem~\ref{corol:geom}}\label{sec:graph}

Let us recall that  $V_0'(M)$ denotes the infimum of the volumes of hyperbolic manifolds which can be embedded in $M$  with and incompressible cusp in $M$ or as complement of a (possibly empty) link in $M$. Since hyperbolic volumes form a well-ordered set, this infimum is in fact a minimum, hence $V_0'(M)>0$.

Let $M$ be a graph manifold. After Cheeger-Gromov
\cite{CheegerGromov1}, one can construct Riemannian metrics on $M$ with sectional curvature pinched between  $-1$ and $1$ whose volume is arbitrarily small. Thus $\overline{R}(M) \geq 0 $. Since $V_0'(M) >0$, this proves the `only if' part of the equivalence.

The `if' part follows from Theorem~\ref{existence sequence} and the following variant of Theorem~\ref{TH:GRAPH}:

\begin{theo}\label{TH:GRAPH2}
Let $M$ be a closed, orientable, non-simply connected, irreducible $3$-ma\-ni\-fold.
Let $g_n$ be a sequence of Riemannian metrics satisfying:
\begin{itemize}
\item[(1)] The sequence $\vol(g_n)$ is bounded.
\item[(2)]  For all  $\varepsilon>0$, if $x_n\in M$ is a sequence such that for all  $n$, $x_n$ is in the $\varepsilon$-thick part of  $(M,g_n)$, then    $(M,g_n,x_n)$ subconverges in the $\mathcal C^2$ topology to  a pointed hyperbolic manifold with volume strictly less than $V'_0(M)$.
\item[(3)] The sequence $(M,g_n)$ has  locally controlled curvature   in the sense of Perelman.
\end{itemize}
Then  $M$ is  a graph manifold.
\end{theo}

\begin{proof}
Let $H^1,\ldots,H^m$ be hyperbolic limits given by Proposition~\ref{partie epaisse}  and let $\varepsilon_n\to 0$ be a sequence  chosen as in the remark after  Proposition~\ref{partie epaisse}, or in the beginning of Section \ref{sec:mince} to describe the local structure of the thin part.
As in the proof of Theorem~\ref{TH:GRAPH2}, for each $i$ we fix a compact core $\bar H^i$ of
$H^i$ and for each
 $n$ a submanifold $\bar H_n^i$ and an approximation $\phi_n^i:\bar
H_n^i \to \bar H^i$.
The fact that the volume of each hyperbolic manifold $H^i$ is les than $V_0'(M)$ implies the following result:

\begin{lem}\label{lem:compressible}
Up to taking a subsequence of $M_n$,  for all $i \in{1,\ldots, m}$ 
each component of
$\partial \bar H_n^{i}$ is compressible in $M$ for all $n$.
\end{lem}

\begin{proof} Indeed if the conclusion of Lemma~\ref{lem:compressible} does not hold, then  up to extracting a subsequence, we may assume that there exists an integer $i_0\in{1,\cdots, m}$ such that
$\partial \bar H_n^{i_0}$ contains an incompressible torus for all $n$. 
From the definition of $V_0'(M)$ this would contradict the inequality  $\vol(H^{i_0}) < V_0'(M)$.
\end{proof}

From this lemma on, the proof of Theorem~\ref{TH:GRAPH2} is identical to the proof of Theorem~\ref{TH:GRAPH}.
\end{proof}

\bibliographystyle{plain}
\bibliography{refs}

\noindent Institut Fourier, Universit\'e de Grenoble I, UMR 5582 CNRS-UJF, 38402, Saint-Martin-d'H\`eres, France

\noindent   laurent.bessieres@ujf-grenoble.fr,  g.besson@fourier.ujf-grenoble.fr

\medskip

\noindent Institut Math\'ematique de Toulouse, UMR 5219 CNRS-UPS, Universit\'e Paul
Sabatier, 31062 Toulouse Cedex 9, France.

\noindent  boileau@picard.ups-tlse.fr

\medskip

\noindent Institut de Recherche Math\'ematique Avanc\'ee, Universit\'e Louis Pasteur, 7 rue Ren\'e Descartes, 67084 Strasbourg Cedex, France

\noindent  maillot@math.u-strasbg.fr

\medskip

\noindent Departament de Matem\`atiques, Universitat Aut\`onoma de Barcelona, E-08193 Bellaterra, Spain

\noindent  porti@mat.uab.es

\end{document}